\documentclass[1p]{elsarticle}

\usepackage{lineno}
\usepackage[colorlinks,linkcolor={blue}]{hyperref}
\modulolinenumbers[5]

\usepackage{amsfonts}
\usepackage{graphicx}
\usepackage{amsmath}
\usepackage{graphicx}
\usepackage{amssymb}

\newtheorem{theorem}{Theorem}

\newtheorem{example}[theorem]{Example}

\newtheorem{lemma}[theorem]{Lemma}

\newtheorem{remark}[theorem]{Remark}
\newenvironment{proof}[1][Proof]{\textbf{#1.} }{\ \rule{0.5em}{0.5em}}

\journal{Journal of Differential Equations}

\begin{document}
\begin{frontmatter}
\title{Beyond the Melnikov method II: multidimensional setting }

\author{Maciej J. Capi\'nski\footnote{Partially supported by NCN grants 2015/19/B/ST1/01454, 2016/21/B/ST1/02453 and by the Faculty of Applied Mathematics AGH UST statutory tasks within subsidy of Ministry of Science and Higher Education.}}
\ead{mcapinsk@agh.edu.pl}
\address{AGH University of Science and Technology, al. Mickiewicza 30, 30-059 Krak\'ow, Poland}

\author{Piotr Zgliczy\'nski\footnote{Partially supported by the NCN grant 2015/19/B/ST1/01454}}
\ead{umzglicz@cyf-kr.edu.pl}
\address{Jagiellonian University, ul. prof. Stanis\l awa \L ojasiewicza 6, 
30-348 Krak\'ow, Poland}

\begin{abstract}
  We present a Melnikov type approach for establishing transversal intersections of stable/unstable manifolds of perturbed normally hyperbolic invariant manifolds. 
 We do not need to know the explicit formulas for the homoclinic orbits prior to the perturbation. We also do not need to compute any integrals along such homoclinics. All needed bounds are established using rigorous computer assisted numerics. Lastly, and most importantly, the method establishes intersections for an explicit range of parameters, and not only for perturbations that are `small enough', as is the case in the classical Melnikov approach.
\end{abstract}

\end{frontmatter}



\section{Introduction}
This paper is a sequel to \cite{Meln}, which developed a tool for establishing the splitting of separatrices for an explicit range of perturbation parameters. The paper \cite{Meln} treated the case of one-dimensional separatrices. In the current work we generalise the results to multidimensional setting. We treat the case of a normally hyperbolic invariant manifold (NHIM), with multi dimensional stable and unstable manifolds. The manifolds do not need to be of the same dimension. Also, the system does not need to be Hamiltonian. The stable and unstable manifolds of the NHIM can coincide prior to the perturbation (or one can be contained in the other in the case of unequal dimensions), and our method ensures that for a given range of parameters, the manifolds will intersect transversally after the perturbation of the system. We also formulate our results so that we can establish transversal splitting after the perturbation in the setting when prior to the perturbation manifolds coincide on some of the coordinates and on others they intersect transversally.

There are two main differences between our result and the more standard Melnikov type methods \cite{Chow,delshamsMelPotential,delshamsLlave,delshamsGuttSpilt2000,HM,M,Tre94,Wi}. The first is that these are based on investigating integrals along homoclinic orbits to NHIMs, and to do so one usually requires to know the formulae for them. In our approach we do not require to know such formulae since we do not need to compute integrals. Our result follows from bounds on the stable and unstable manifolds and on their dependence under the perturbation. There are a number of methods which allow for establishing such bounds using computer assisted tools \cite{BJLM,Jay,Geom,FH,LJR,JM}. Instead of computing integrals we use such estimates, combined with the Brouwer degree, to establish explicit conditions under which the manifolds intersect transversally. The second difference is that the classical Melnikov type methods establish the intersection of manifolds for sufficiently small perturbations, but our method establishes such intersections for an explicitly given range of parameters.

The paper is organised as follows. Section \ref{sec:preliminaries} contains preliminaries. In section \ref{sec:zeros} we state our main results in abstract setting. We discuss there how to establish zeros of functions under perturbation. We formulate results which imply non degenerate zeros, and give their proofs using the Brouwer degree. In section \ref{sec:Melnikov} we present several scenarios of intersections of separatrices of NHIMs, and prove their existence using the tools established in section \ref{sec:zeros}. In section \ref{sec:verif} we discuss how to verify the needed assumptions. Finally, section \ref{sec:example} contains an example of application. There we treat a system which has two dimensional stable and unstable manifolds of a hyperbolic fixed point, which coincide prior to the perturbation, but  intersect transversally after.


\section{Preliminaries}\label{sec:preliminaries}

\subsection{Notations}

We use the notation $B_{k}(x,r)$ for a ball of radius $r$, centred at $x$ in
$\mathbb{R}^k$. If not stated otherwise, the ball is under the Euclidean norm.
We also use a simplified notation $B_{k}:=B_{k}(0,1)$. For a set $U\subset \mathbb{R}^k$ we will write $\mathrm{int}U$ for its interior and $\overline{U}$ for its closure.

For a matrix $A \in \mathbb{R}^{n \times n}$ and norm $\|\cdot\|$ on $\mathbb{R}^n$ we define%
\[
m\left(  A\right)  =\left\{
\begin{array}
[c]{lll}%
\frac{1}{\left\Vert A^{-1}\right\Vert } &  & \text{if det}A\neq 0\\
0 &  & \text{otherwise.}%
\end{array}
\right.
\]
We remark that for any vector $x$ holds%
\[
\left\Vert Ax\right\Vert \geq m\left(  A\right)  \left\Vert x\right\Vert .
\]

If $W$ is a manifold and $p \in W$ then by $T_p W$ we will denote the tangent space to $W$ at the point $p$.

\subsection{Properties of the local Brouwer degree}

Let $f:\mathbb{R}^{n}\rightarrow\mathbb{R}^{n}$ be a continuous function and
let $U\subset\mathbb{R}^{n}$ be an open set. For $c\in\mathbb{R}^{n}$ we use
the notation $\deg(f,U,c)$ to stand for the local Brouwer degree of $f$ in $U$
at $c$. The local Brouwer degree has the following properties (see \cite[Ch. III]{Sch} for proofs)

\textbf{Solution property:} If $\deg\left(  f,U,c\right)  \neq0$, then there
exists an $x\in U$ for which
$f(x)=c.$

\textbf{Homotopy property:} Assume that $H:\left[  0,1\right]  \times
U\rightarrow\mathbb{R}^{n}$ is continuous and%
\begin{equation}
\bigcup_{\lambda\in\left[  0,1\right]  }H_{\lambda}^{-1}\left(  c\right)  \cap
U\quad\text{is compact} \label{eq:hom-prop-cond}%
\end{equation}
then for all $\lambda\in\left[  0,1\right]  $%
\[
\deg\left(  H_{\lambda},U,c\right)  =\deg\left(  H_{0},U,c\right),
\]
where $H_\lambda(x)=H(\lambda,x)$.

\begin{remark}
\label{rem:homotopy-prop}Condition $c\notin H\left(  \left[  0,1\right]
,\partial U\right)  $ implies (\ref{eq:hom-prop-cond}).
\end{remark}

\textbf{Degree property for affine maps:} If $f(x)=B(x-x_{0})+c$, where $B$ is
an invertible matrix, $x_{0},c\in\mathbb{R}^{n}$ and $x_{0}\in U$, then
\[
\deg(f,U,c)=\mathrm{sgn}\det B.
\]

\subsection{Interval Newton method}

We start by writing out the interval arithmetic notations conventions that
will be used in the paper. Let $U$ be a subset of $\mathbb{R}^{k}$. We shall
denote by $[U]$ an interval enclosure of the set $U$, that is, a set
\[
\lbrack U]=\Pi_{i=1}^{k}[a_{i},b_{i}]\subset\mathbb{R}^{k},
\]
such that $U\subset\lbrack U]$. Similarly, for a family of matrixes
$\mathbf{A}\subset\mathbb{R}^{k\times m}$ we denote its interval enclosure as
$\left[  \mathbf{A}\right]  $, that is, a set
\[
\left[  \mathbf{A}\right]  =\left(  [a_{ij},b_{ij}]\right)
_{\substack{i=1,...,k\\j=1,...,m}}\subset\mathbb{R}^{k\times m},
\]
such that $\mathbf{A}\subset\left[  \mathbf{A}\right]  $. For $f:\mathbb{R}%
^{k}\rightarrow\mathbb{R}^{m}$, by $[Df(U)]$ we shall denote an interval
enclosure
\[
\lbrack Df(U)]=\left[  \left\{  A\in\mathbb{R}^{k\times m}|A_{ij}\in\left[
\inf_{x\in U}\frac{\partial f_{i}}{\partial x_{j}}(x),\sup_{x\in U}%
\frac{\partial f_{i}}{\partial x_{j}}(x)\right]  \right\}  \right]  .
\]
For a set $U$ and a family of matrixes $\mathbf{A}$ we shall use the notation
$\left[  \mathbf{A}\right]  \left[  U\right]  $ to denote an interval
enclosure
\[
\left[  \mathbf{A}\right]  \left[  U\right]  =\left[  \left\{  Au:A\in\left[
\mathbf{A}\right]  ,u\in\left[  U\right]  \right\}  \right]  .
\]
We shall say that a family of matrixes $\mathbf{A}\subset\mathbb{R}^{k\times
k}$ is invertible, if each matrix $A\in\mathbf{A}$ is invertible. We shall
also use the notation%
\[
\left[  \mathbf{A}\right]  ^{-1}\left[  U\right]  =\left[  \left\{
A^{-1}u:A\in\left[  \mathbf{A}\right]  ,u\in\left[  U\right]  \right\}
\right]  .
\]

Let now us consider now a $C^r$ function
\[
f:\mathbb{R}^{k}\times\mathbb{R}^{m}\rightarrow\mathbb{R}^{m}.
\]
Below we present an interval Newton type method \cite{A,Mo,N} for establishing estimates on
the set $\left\{  f=0\right\}  $. 

Consider $x\in\mathbb{R}^{k}$ and define a function $f_{x}:\mathbb{R}%
^{m}\rightarrow\mathbb{R}^{m}$ as
\[
f_{x}\left(  y\right)  :=f\left(  x,y\right)  .
\]
For $X\subset\mathbb{R}^{k}$ and $Y\subset\mathbb{R}^{m}$, by $Df_{X}\left(
Y\right)  $ we denote the family of matrixes
\[
Df_{X}\left(  Y\right)  =\left\{  D\left(  f_{x}\right)  \left(  y\right)
:x\in X,y\in Y\right\}  .
\]
Bounds on $\left\{  f=0\right\}  $ can be obtained by using the interval
Newton method. Below theorem is a well known modification (see for instance
\cite[p. 376]{Rump}) of the method, that includes a parameter.

\begin{theorem}
\label{th:interval-Newton}Let $X=\Pi_{i=1}^{k}\left[  a_{i},b_{i}\right]
\subset\mathbb{R}^{k}$ and $Y=\Pi_{i=1}^{m}\left[  c_{i},d_{i}\right]
\subset\mathbb{R}^{m}$. Consider $y_{0}\in\mathrm{int}Y$ and%
\[
N\left(  y_{0},X,Y\right)  =y_{0}-\left[  Df_{X}\left(  Y\right)  \right]
^{-1}\left[  f_{X}\left(  y_{0}\right)  \right]  .
\]
If
\[
N\left(  y_{0},X,Y\right)  \subset\mathrm{int}Y,
\]
then there exists a $C^r$ function $q:X\rightarrow Y$ such that $f\left(  x,q\left(
x\right)  \right)  =0.$

\end{theorem}


\section{Existence of zeros of functions under perturbation\label{sec:zeros}}

Let $x\in\mathbb{R}^{k}$, $y:\mathbb{R}\times\mathbb{R}^{k}\rightarrow
\mathbb{R}^{k}$ and let $U\subset\mathbb{R}^{k}$ be an open set. In this
section we will formulate conditions under which for any $\varepsilon
\in(0,\epsilon]$ there exists an $x=x\left(  \varepsilon\right)  \in U$ such
that $y\left(  \varepsilon,x\right)  =0$. We shall also investigate conditions
that will ensure that the zero is non degenerate for a given fixed
$\varepsilon\in(0,\epsilon].$

Our conditions will be based on the following lemma, which is a direct
consequence of the solution property of the Brouwer degree.

\begin{lemma}
\label{lem:zero-from-degree}If for any $\varepsilon\in(0,\epsilon]$%
\[
\deg\left(  y\left(  \varepsilon,\cdot\right)  ,U,0\right)  \neq0,
\]
then for any $\varepsilon\in(0,\epsilon]$ there exists an $x=x\left(
\varepsilon\right)  \in U$ such that $y\left(  \varepsilon,x\right)  =0$.
\end{lemma}

The result can readily be applied if $x\rightarrow y\left(  0,x\right)  $ has
non-degenerate zero, but this is not what will be our objective here. We will
want to formulate results that guarantee the existence of zeros of $y$ under
perturbation in the case when zeros of $y(0,x)$ are degenerate prior to the perturbation.

Making above more precise, we will assume that $k=k_{1}+k_{2}$, that
\begin{equation}
y=\left(  y_{1},y_{2}\right)  :\mathbb{R}\times\mathbb{R}^{k_{1}}%
\times\mathbb{R}^{k_{2}}\rightarrow\mathbb{R}^{k_{1}}\times\mathbb{R}^{k_{2}},
\label{eq:y1-y2-setting}%
\end{equation}
and our discussion will be under the assumption that for an open set
$U\subset\mathbb{R}^{k_{1}}\times\mathbb{R}^{k_{2}}$, for any $x\in U$, we
have%
\begin{equation}
y_{2}\left(  0,x\right)  =0. \label{eq:y2-zero}%
\end{equation}

On $\mathbb{R}^{k}_{1}\times\mathbb{R}^{k_{2}}$ we will use a norm given by
$\|(x_{1},x_{2})\|=\max(\|x_{1}\|,\|x_{2}\|)$ for some norms on $\mathbb{R}%
^{k_{1}}$ and $\mathbb{R}^{k_{2}}$. We assume that $y$ is $C^{2}$.

Above we assume that on the $y_{2}$ coordinate the function $y|_{U}$ is  zero.
On the $y_{1}$ coordinate, as of yet, we have not made any assumptions. In
practice, to obtain the existence of $x\left(  \varepsilon\right)  $ for which
$y\left(  \varepsilon,x(\varepsilon)\right)  =0$, on the $y_{1}$ coordinate we
will need to have a non-degenerate zero before the perturbation. The
perturbation should be small enough so that this zero will survive.

Our objective will be to formulate lemmas that will imply that assumptions of
Lemma \ref{lem:zero-from-degree} are fulfilled, in the presence of
(\ref{eq:y2-zero}). We start with the following:

\begin{lemma}
Let $E=\left[  0,\epsilon\right]  $. Assume that (\ref{eq:y2-zero}) is
satisfied and for any $x\in\partial U$ holds
\begin{equation}
0\notin\left[  \left(  y_{1},\frac{\partial y_{2}}{\partial\varepsilon
}\right)  \left(  E,x\right)  \right]  \label{eq:zero-not-on-boundary}%
\end{equation}
and $\deg\left(  \left(  y_{1},\frac{\partial y_{2}}{\partial\varepsilon
}\right)  \left(  0,\cdot\right)  ,U,0\right)  \neq0$.

Then for any $\varepsilon\in(0,\epsilon]$
\[
\deg\left(  y\left(  \varepsilon,\cdot\right)  ,U,0\right)  \neq0.
\]

\end{lemma}

\begin{proof}
The proof will be based on the homotopy property of the local Brouwer degree.
First we observe that from (\ref{eq:y2-zero}) we obtain%
\[
y_{2}\left(  \varepsilon,x\right)  =y_{2}\left(  \varepsilon,x\right)
-y_{2}\left(  0,x\right)  =\int_{0}^{1}\frac{d}{du}y_{2}\left(  u\varepsilon
,x\right)  dy=\varepsilon\int_{0}^{1}\frac{\partial y_{2}}{\partial
\varepsilon}\left(  u\varepsilon,x\right)  dy.
\]
For fixed $\varepsilon\in(0,\epsilon]$ we consider the following homotopy%
\[
H\left(  \lambda,x\right)  =\left(  y_{1}\left(  \lambda\varepsilon,x\right)
,\left(  1+\lambda\left(  \varepsilon-1\right)  \right)  \int_{0}^{1}%
\frac{\partial y_{2}}{\partial\varepsilon}\left(  \lambda u\varepsilon
,x\right)  dy\right)  .
\]
We have chosen such homotopy since $H\left(  1,x\right)  =y\left(
\varepsilon,x\right)  $ and $H(0,x)=\left(  y_{1},\frac{\partial y_{2}%
}{\partial\varepsilon}\right)  \left(  0,x\right)  .$

Since $\int_{0}^{1}\frac{\partial y_{2}}{\partial\varepsilon}\left(
u\varepsilon,x\right)  dy\in\left[  \frac{\partial y_{2}}{\partial\varepsilon
}\left(  E,x\right)  \right]  ,$ by (\ref{eq:zero-not-on-boundary}) we see
that for $x\in\partial U$%
\[
H\left(  \lambda,x\right)  \neq0.
\]
By the homotopy property and Remark \ref{rem:homotopy-prop} we obtain%
\begin{align*}
\deg\left(  y\left(  \varepsilon,\cdot\right)  ,U,0\right)   &  =\deg\left(
H\left(  1,\cdot\right)  ,U,0\right) \\
&  =\deg\left(  H\left(  0,\cdot\right)  ,U,0\right)  =\deg\left(  \left(
y_{1},\frac{\partial y_{2}}{\partial\varepsilon}\right)  \left(
0,\cdot\right)  ,U,0\right)  \neq0,
\end{align*}
as required.
\end{proof}

We now formulate a more explicit result that can be verified in practice:

\begin{lemma}
\label{lem:y-zero-practical} Let $y$ be as above.

Let $p=(p_{1},p_{2}) \in\mathbb{R}^{k_{1}} \times\mathbb{R}^{k_{2}}$ be such
that $y(0,p)=0$ and $R>0$.

Assume that $U=B\left( p,R\right)  $ and let $E=\left[  0,\epsilon_{0}\right]
$. Let%
\[
A=\left(
\begin{array}
[c]{cc}%
A_{11} & A_{12}\\
A_{21} & A_{22}%
\end{array}
\right)  :=\left(
\begin{array}
[c]{cc}%
\frac{\partial y_{1}}{\partial x_{1}}\left(  0,p\right)  & \frac{\partial
y_{1}}{\partial x_{2}}\left(  0,p\right) \\
\frac{\partial^{2}y_{2}}{\partial\varepsilon\partial x_{1}}\left(  0,p\right)
& \frac{\partial^{2}y_{2}}{\partial\varepsilon\partial x_{2}}\left(
0,p\right)
\end{array}
\right)
\]
and%
\begin{align*}
\Delta_{1}  &  =\left[  \frac{\partial y_{1}}{\partial x}\left(  E,U\right)
-(A_{11},0)\right] ,\\
\Delta_{2}  &  =\left[ \frac{\partial^{2}y_{2}}{\partial\varepsilon\partial
x}\left(  E,U\right)  -(0,A_{22})\right] .
\end{align*}
If%
\begin{align}
m\left(  A_{11}\right)  R  &  >\epsilon\left\Vert \frac{\partial y_{1}%
}{\partial\varepsilon}\left(  E,p\right)  \right\Vert +\left\Vert \Delta
_{1}\right\Vert R,\label{eq:estmA11}\\
m\left(  A_{22}\right)  R  &  >\left\Vert \frac{\partial y_{2}}{\partial
\varepsilon}\left(  E,p\right)  \right\Vert +\left\Vert \Delta_{2}\right\Vert
R,\label{eq:estmA22}%
\end{align}
then
\[
\deg\left( y\left(  \varepsilon,\cdot\right)  ,U,0\right)  \neq0.
\]

\end{lemma}

\begin{proof}
By (\ref{eq:y2-zero}) we see that for any $x\in\overline{U},$ $\frac{\partial
y_{2}}{\partial x}\left(  0,x\right)  =0$. For any $x\in\overline{U}$,%
\begin{align}
y_{2}\left(  \varepsilon,x\right)  -y_{2}\left(  \varepsilon,p\right)   &
=\int_{0}^{1}\frac{d}{ds}y_{2}\left(  \varepsilon,p+s\left(  x-p\right)
\right)  ds\label{eq:y1-difference}\\
&  =\int_{0}^{1}\frac{\partial y_{2}}{\partial x}\left(  \varepsilon
,p+s\left(  x-p\right)  \right)  ds\cdot\left(  x-p\right) \nonumber\\
&  =\varepsilon\int_{0}^{1}\int_{0}^{1}\frac{\partial^{2}y_{2}}{\partial
\varepsilon\partial x}\left(  u\varepsilon,p+s\left(  x-p\right)  \right)
du\,ds\left(  x-p\right)  .\nonumber
\end{align}
Also%
\begin{align}
y_{1}\left(  \varepsilon,x\right)  -y_{1}\left(  \varepsilon,p\right)   &
=\int_{0}^{1}\frac{d}{ds}y_{1}\left(  \varepsilon,p+s\left(  x-p\right)
\right)  ds\label{eq:y2-difference}\\
&  =\int_{0}^{1}\frac{\partial y_{1}}{\partial x}\left(  \varepsilon
,p+s\left(  x-p\right)  \right)  ds\cdot\left(  x-p\right)  .\nonumber
\end{align}
We will also use the fact that%
\begin{equation}
y\left(  \varepsilon,p\right)  =\varepsilon\int_{0}^{1}\frac{\partial
y}{\partial\varepsilon}\left(  u\varepsilon,p\right)  du. \label{eq:y0-bd}%
\end{equation}

Consider now a homotopy%
\[
H\left(  \lambda,x\right)  =\left(
\begin{array}
[c]{c}%
H_{1}\left(  \lambda,x\right) \\
H_{2}\left(  \lambda,x\right)
\end{array}
\right)  ,
\]
defined as%
\begin{align*}
H_{1}\left(  \lambda,x\right)  =  & \varepsilon\lambda\int_{0}^{1}%
\frac{\partial y_{1}}{\partial\varepsilon}\left(  u\varepsilon,p\right)  du  +\left(  \int_{0}^{1}\frac{\partial y_{1}}{\partial x}\left(
\lambda\varepsilon,p+\lambda s\left(  x-p\right)  \right)  ds\right)  \left(
x-p\right)  ,\\
H_{2}\left(  \lambda,x\right)  =  & \varepsilon\left[  \lambda\int_{0}%
^{1}\frac{\partial y_{2}}{\partial\varepsilon}\left(  u\varepsilon,p\right)
du\right.   \left.  +\left(  \int_{0}^{1}\int_{0}^{1}\frac{\partial^{2}y_{2}}%
{\partial\varepsilon\partial x}\left(  \lambda u\varepsilon,p+\lambda s\left(
x-p\right)  \right)  du\,ds\right)  \left(  x-p\right)  \right]  .
\end{align*}
We have chosen such homotopy since from (\ref{eq:y1-difference}%
--\ref{eq:y0-bd}) we see that%
\[
H\left(  1,x\right)  =y\left(  \varepsilon,x\right)  ,
\]
and%
\[
H\left(  0,x\right)  =\left(
\begin{array}
[c]{rr}%
\frac{\partial y_{1}}{\partial x_{1}}\left(  0,p\right)  & \frac{\partial
y_{1}}{\partial x_{2}}\left(  0,p\right) \\
\varepsilon\frac{\partial^{2}y_{2}}{\partial\varepsilon\partial x_{1}}\left(
0,p\right)  & \varepsilon\frac{\partial^{2}y_{2}}{\partial\varepsilon\partial
x_{2}}\left(  0,p\right)
\end{array}
\right)  (x-p)=\left(
\begin{array}
[c]{cc}%
A_{11} & A_{12}\\
\varepsilon A_{12} & \varepsilon A_{22}%
\end{array}
\right)  (x-p).
\]

We now will want to use the homotopy property to ensure that the degree is
defined and is the same along the homotopy $H$. For this we need to check that
$H\left(  \lambda,\partial U\right)  \neq0$ (see Remark
\ref{rem:homotopy-prop}). Take $x\in\partial U$, meaning that $\left\Vert
x-p\right\Vert =R$.

We have two cases $\|p_{1} - x_{1}\|=R$ and $\|p_{2} - x_{2}\|=R$.

If $\|p_{1} - x_{1}\|=R$, then
\begin{align*}
&  \left\Vert H_{1}\left(  \lambda,x\right)  \right\Vert \\
&  =\left\Vert \varepsilon\lambda\int_{0}^{1}\frac{\partial y_{1}}
{\partial\varepsilon}\left(  u\varepsilon,p\right)  du+\int_{0}^{1}%
\frac{\partial y_{1}}{\partial x}\left(  \lambda\varepsilon,p+\lambda s\left(
x-p\right)  \right)  ds\left(  x-p\right)  \right\Vert \\
&  \geq-\epsilon\left\Vert \frac{\partial y_{1}}{\partial\varepsilon}\left(
E,p\right)  \right\Vert +m\left(  A_{11}\right)  R-\left\Vert \int_{0}%
^{1}\frac{\partial y_{1}}{\partial x}\left(  \lambda\varepsilon,p+\lambda
s\left(  x-p\right)  \right)  -(A_{11},0) ds \right\Vert R\\
&  =-\epsilon\left\Vert \frac{\partial y_{1}}{\partial\varepsilon}\left(
E,p\right)  \right\Vert +m\left(  A_{11}\right)  R-\left\Vert \Delta
_{1}\right\Vert R >0.
\end{align*}

If $\|p_{2} -x_{2}\|=R$, then
\begin{align*}
\left\Vert \frac{1}{\varepsilon}H_{2}\left(  \lambda,x\right)  \right\Vert  &
=\left\Vert \lambda\int_{0}^{1}\frac{\partial y_{2}}{\partial\varepsilon
}\left(  u\varepsilon,p\right)  du\right. \\
&  \quad\left.  +\left(  \int_{0}^{1}\int_{0}^{1}\frac{\partial^{2}y_{2}%
}{\partial\varepsilon\partial x}\left(  \lambda u\varepsilon,p+\lambda
s\left(  x-p\right)  \right)  du\,ds\right)  \left(  x-p\right)  \right\Vert
\\
&  \geq-\left\Vert \frac{\partial y_{2}}{\partial\varepsilon}\left(
E,p\right)  \right\Vert +m\left(  A_{22}\right)  R\\
&  \quad-\left\Vert \int_{0}^{1}\int_{0}^{1}\frac{\partial^{2}y_{2}}%
{\partial\varepsilon\partial x}\left(  \lambda u\varepsilon,p+\lambda s\left(
x-p\right)  \right)  du\,ds-(0,A_{22})\right\Vert R\\
&  =-\left\Vert \frac{\partial y_{2}}{\partial\varepsilon}\left(  E,p\right)
\right\Vert +m\left(  A_{22}\right)  R-\left\Vert \Delta_{2}\right\Vert R >0.
\end{align*}

Observe that $H(0,x) \neq0$ for $x \in\partial U=\partial B(p_{0},R) $ implies
that $A$ is invertible.

Thus, by the homotopy property and the degree for affine maps
\[
\deg\left(  y\left(  \varepsilon,\cdot\right)  ,U,0\right)  =\deg\left(
H_{1}\left(  \cdot\right)  ,U,0\right)  =\deg\left(  H_{0}\left(
\cdot\right)  ,U,0\right)  =\mathrm{sgn}\det A \neq0
\]
as required.
\end{proof}

\subsection{Transversal zeros}

In the applications we have in mind the function $y$ will be obtained as
difference of two graphs of functions representing manifolds, i.e.
\begin{equation}
y(\varepsilon,x)= w(\varepsilon,x) - v(\varepsilon,x).\label{eq:y=diff}%
\end{equation}

Let us fix an open set $U \subset\mathbb{R}^{k_{1}+k_{2}}$ and let
$W_{\varepsilon}=\{(x,w(\varepsilon,x)), \quad x \in U\}$ and $V_{\varepsilon
}=\{(x,v(\varepsilon,x)), \quad x\in U\}$.

\begin{lemma}
\label{lem:trans}  Consider the same assumptions as in Lemma~\ref{lem:y-zero-practical}
concerning function $y$ given by (\ref{eq:y=diff}).
Then 
\begin{itemize}
\item any matrix $M
\in\left[  \frac{\partial y}{\partial x}(E,U)\right] $ is an isomorphism
\item for every $\varepsilon\in(0,\epsilon_{0} ]$ the manifolds $W_{\varepsilon
}$ and $V_{\varepsilon}$ have a unique intersection which is transversal.
\end{itemize}
\end{lemma}

\begin{proof}
The existence of the intersection follows from
Lemma~\ref{lem:y-zero-practical}.

For the uniqueness and transversality let us consider  any matrix $M
\in\left[  \frac{\partial y}{\partial x}(E,U)\right] $. Since our $x$ and $y$
variables are split in two components i.e. $y(x)=(y_{1}(x_{1},x_{2}%
),y_{2}(x_{1},x_{2}))$  we can also characterise matrix $M$ by its components
$M_{ij}$ satisfying $M_{ij} \in\frac{\partial y_{i}}{\partial x_{j}}(E,U)$.

We see that
\[
\left(  M_{11},M_{12}\right)  -\left(  A_{11},0\right)  \in\left[
\frac{\partial y_{1}}{\partial x}(E,U)\right]  -\left(  \frac{\partial y_{1}%
}{\partial x_{1}}(0,p),0\right)  =\Delta_{1}.
\]
Since $y_{2}\left(  \varepsilon=0,\cdot\right)  |_{U}\equiv0$ we see that for
any $\varepsilon\in E$ and $x\in U$%
\[
\frac{\partial y_{2}}{\partial x}(\varepsilon,x)=\frac{\partial y_{2}%
}{\partial x}(0,x)+\int_{0}^{1}\frac{d}{ds}\frac{\partial y_{2}}{\partial
x}(s\varepsilon,x)=\varepsilon\int_{0}^{1}\frac{\partial^{2}y_{2}}%
{\partial\varepsilon\partial x}(s\varepsilon,x)\in\varepsilon\left[
\frac{\partial^{2}y_{2}}{\partial x\partial\varepsilon}(E,U)\right]  ,
\]
so
\begin{multline*}
\left(  M_{21},M_{22}\right)  -\varepsilon\left(  0,A_{22}\right)  \in\left[
\frac{\partial y_{2}}{\partial x}(E,U)\right]  -\varepsilon\left(
0,\frac{\partial^{2}y_{2}}{\partial\varepsilon\partial x_{2}}(0,p)\right)  \\
\subset\varepsilon\left(  \left[  \frac{\partial^{2}y_{2}}{\partial
x\partial\varepsilon}(E,U)\right]  -\left(  0,\frac{\partial^{2}y_{2}%
}{\partial\varepsilon\partial x_{2}}(0,p)\right)  \right)  =\varepsilon
\Delta_{2}.
\end{multline*}

Observe that from (\ref{eq:estmA11},\ref{eq:estmA22}) it follows that
\[
m(A_{11})>\Vert\Delta_{1}\Vert,\qquad m(A_{22})>\Vert\Delta_{2}\Vert.
\]

We are now ready to show that $M$ is an isomorphism. For this it is enough to
show that for any $x$ for which $\max\left(  \Vert x_{1}%
\Vert,\Vert x_{2}\Vert\right)  =1$ holds $Mx\neq0$.

If $\|x_{1}\|=1$, then
\begin{align*}
\|(Mx)_{1}\|  & = \|A_{11}x_{1} + \left( (M_{11},M_{12}) - (A_{11},0)\right) x
\| \geq\|A_{11}x_{1}\| - \|\Delta_{1}\| \geq\\
&   m(A_{11}) - \|\Delta_{1}\| >0.
\end{align*}

If $\Vert x_{2}\Vert=1$, then
\[
\Vert(Mx)_{2}\Vert=\Vert\varepsilon A_{22}x_{2}+\left(  (M_{21},M_{22}%
)-\varepsilon\left(  0,A_{22}\right)  \right)  x\Vert\geq\varepsilon\left(
m(A_{22})-\Vert\Delta_{2}\Vert\right)  >0.
\]
Therefore we see that any $M\in\left[  \frac{\partial y}{\partial
x}(E,U)\right]  $ is an isomorphism.

By changing a coordinate system $(x,y)\mapsto(x,y-v(\varepsilon,x))$ we obtain
that $V_{\varepsilon}=\{(x,0),\quad x\in U\}$ and $W_{\varepsilon
}=\{(x,w(\varepsilon,x)-v(\varepsilon,x)),\quad x\in U\}$.

To establish the transversality observe that
\begin{align*}
T_{(x,0)}V_{\varepsilon}  & =\{(z,0)\quad z\in\mathbb{R}^{k_{1}+k_{2}}\},\\
T_{(x,y(\varepsilon,x))}W_{\varepsilon}  & =\left\{  \left(  z,\frac{\partial
y}{\partial x}(\varepsilon,x)z\right)  \quad z\in\mathbb{R}^{k_{1}+k_{2}%
}\right\}  .
\end{align*}
The intersection is transversal if $\frac{\partial y}{\partial x}%
(\varepsilon,x)$ is in isomorphism. We know that any $M\in\left[
\frac{\partial y}{\partial x}(E,U)\right]  $ is an isomorphism. Since
$\frac{\partial y}{\partial x}(\varepsilon,x)\in\left[  \frac{\partial
y}{\partial x}(E,U)\right]  $, then this is the case.

For the uniqueness observe that
\begin{align*}
y\left(  \varepsilon,z_{1}\right)  -y\left(  \varepsilon,z_{2}\right)    &
=\int_{0}^{1}\frac{d}{dt}y\left(  \varepsilon,z_{2}+t\left(  z_{1}%
-z_{2}\right)  \right)  dt\\
& =\int_{0}^{1}\frac{\partial y}{\partial x}\left(  \varepsilon,z_{2}+t\left(
z_{1}-z_{2}\right)  \right)  dt(z_{1}-z_{2}).
\end{align*}
Since $\int_{0}^{1}\frac{\partial y}{\partial x}\left(  \varepsilon
,z_{2}+t\left(  z_{1}-z_{2}\right)  \right)  dt\in\left[  \frac{\partial
y}{\partial x}(E,U)\right]  $, hence it is an isomorphism. Therefore $y\left(
\varepsilon,z_{1}\right)  -y\left(  \varepsilon,z_{2}\right)  \neq0$ if
$z_{1}\neq z_{2}$.
\end{proof}


\section{Melnikov type results \label{sec:Melnikov}}

In this section we discuss how to apply the results from section
\ref{sec:zeros} to detect intersections of invariant manifolds under
perturbation. The narrative will start from the simplest setting of a
hyperbolic fixed point, and move gradually through cases of increasing generality.

Before we proceed, we set up notations. We consider a smooth family of ODEs,%
\begin{equation}
q^{\prime}=f_{\varepsilon}\left(  q\right)  , \label{eq:ode-family}%
\end{equation}
parameterised by a one dimensional parameter $\varepsilon\in E:=\left[
0,\epsilon\right]  $. We use the notation $\Phi_{t}^{\varepsilon}$ for the
flow induced by (\ref{eq:ode-family}). We assume that the state space is of
dimension $n$, meaning that $f_{\varepsilon}:\mathbb{R}^{n}\rightarrow
\mathbb{R}^{n}$.

For $\varepsilon=0$ we assume that (\ref{eq:ode-family}) has a normally
hyperbolic invariant manifold $\Lambda_{0}$ in $\mathbb{R}^{n}$. We assume
that this manifold survives for all $\varepsilon\in E$ and is perturbed to
$\Lambda_{\varepsilon}$. We shall use $c$ to denote the dimension of
$\Lambda_{\varepsilon}$. We use the notations $W_{\Lambda_{\varepsilon}}^{cu}$
and $W_{\Lambda_{\varepsilon}}^{cs}$ to denote the center-unstable and
center-stable sets to $\Lambda_{\varepsilon}$, respectively:%
\begin{align*}
W_{\Lambda_{\varepsilon}}^{cu}  &  =\left\{  p:\lim_{t\rightarrow-\infty
}\mathrm{dist}\left(  \Phi_{t}^{\varepsilon}\left(  p\right)  ,\Lambda
_{\varepsilon}\right)  =0\right\}  ,\\
W_{\Lambda_{\varepsilon}}^{cs}  &  =\left\{  p:\lim_{t\rightarrow+\infty
}\mathrm{dist}\left(  \Phi_{t}^{\varepsilon}\left(  p\right)  ,\Lambda
_{\varepsilon}\right)  =0\right\}  .
\end{align*}
We assume that the dimension of the unstable coordinate of $\Lambda
_{\varepsilon}$ is $u$, and that the dimension of the stable coordinate is
$s$. This means that $W_{\Lambda_{\varepsilon}}^{cu}$ is of dimension $c+u$
and that $W_{\Lambda_{\varepsilon}}^{cs}$ is of dimension $c+s$.

We consider a certain fixed neighbourhood $U$ of $\Lambda$, within which we
will assume that we have parameterisations of class $C^{2}$ of local
center-unstable and local center-stable manifolds. To be more precise, we
assume that we have two $C^{2}$ functions%
\begin{align*}
w^{cu}  &  :E\times B_{u}\times\Lambda\rightarrow\mathbb{R}^{n},\\
w^{cu}  &  :E\times B_{s}\times\Lambda\rightarrow\mathbb{R}^{n},
\end{align*}
such that:

\begin{enumerate}
\item If $p\in W_{\Lambda_{\varepsilon}}^{cu}\cap U$ and $\Phi_{t}%
^{\varepsilon}\left(  p\right)  \in U$ for all $t\leq0$, then $p=w^{cu}\left(
\varepsilon,\mathrm{cu}\right)  $ for some $\left(  \varepsilon,\mathrm{cu}%
\right)  \in E\times\left(  B_{u}\times\Lambda\right)  $.

\item If $p\in W_{\Lambda_{\varepsilon}}^{cs}\cap U$ and $\Phi_{t}%
^{\varepsilon}\left(  p\right)  \in U$ for all $t\geq0$, then $p=w^{cs}\left(
\varepsilon,\mathrm{cs}\right)  $, for some $\left(  \varepsilon
,\mathrm{cs}\right)  \in E\times\left(  B_{s}\times\Lambda\right)  $.
\end{enumerate}

Here we use the notation $\mathrm{cu}$ (and $\mathrm{cs}$) to stand for the
parameter of the center-unstable (and center-stable) coordinates.

\begin{remark}
In the case of a non autonomous system $q^{\prime}=f_{\varepsilon}(t,q)$ we
can extend the state space to include the time and treat the time (in the
extended phase space) as a central coordinate. This means that above setting
can also be applied to non autonomous systems.
\end{remark}

In subsequent sections we will show how to reduce the problem of finding
intersections of manifolds to finding zeros of functions. In other words, we
will reduce the problem to the setting from section \ref{sec:zeros}. Depending
on the system we can have different setting of the manifolds prior to
perturbation. On one extreme end, for $\varepsilon=0$ we could have two
coinciding manifolds, meaning that $W_{\Lambda_{0}}^{cu}=W_{\Lambda_{0}}^{cs}%
$. On the other end, we could have transversal intersections of $W_{\Lambda
_{0}}^{cu}$, $W_{\Lambda_{0}}^{cs}$. In between is the case where on some
coordinates the manifolds $W_{\Lambda_{0}}^{cu}$ and $W_{\Lambda_{0}}^{cs}$
coincide, and in other coordinates their intersections are transversal. A
typical example of such setting would be a non fully integrable Hamiltonian
system with several integrals of motion. The integrals of motion are the
coordinates on which the manifolds coincide, but on other coordinates their
intersections could be transversal. The particular case of how the manifolds
$W_{\Lambda_{0}}^{cu}$ and $W_{\Lambda_{0}}^{cs}$ intersect prior to the
perturbation will determine the dimensions $k_{1}$ and $k_{2}$ in
(\ref{eq:y1-y2-setting}). If the manifolds coincide, i.e. $W_{\Lambda_{0}%
}^{cu}=W_{\Lambda_{0}}^{cs}$, then $k_{1}=0$ and $k_{2}=k$. If the
intersection of the manifolds is transversal, then $k_{1}=k$ and $k_{2}=0$.
Finally in the setting where we have $l$ integrals of motion and in the
remaining coordinates we have transversal intersection, $k_{1}=k-l$ and
$k_{2}=l$.

We now discuss several cases in which we show how the approach from section
\ref{sec:zeros} can be applied to prove transversal intersections of
$W_{\Lambda_{\varepsilon}}^{cu}$ and $W_{\Lambda_{\varepsilon}}^{cs}$. We
start with the simplest setting and then build up the generality as we
progress through the section.

\subsection{The case of a hyperbolic fixed point with stable/unstable
manifolds of equal dimension\label{sec:fixed-pt}}

In this section we assume that $\Lambda_{\varepsilon}=\left\{  \lambda
_{\varepsilon}\right\}  $ consists of a family of hyperbolic fixed points. In
this reduced setting we shall write $W_{\Lambda_{\varepsilon}}^{u}$,
$W_{\Lambda_{\varepsilon}}^{s},$ $w^{u}\left(  \varepsilon,\mathrm{u}\right)
$ and $w^{s}\left(  \varepsilon,\mathrm{s}\right)  $ instead of $W_{\Lambda
_{\varepsilon}}^{cu}$, $W_{\Lambda_{\varepsilon}}^{cs},$ $w^{cu}\left(
\varepsilon,\mathrm{cu}\right)  $ and $w^{cs}\left(  \varepsilon,\mathrm{cs}%
\right)  $, respectively. (This is because there is no `center' coordinate to
consider, and the center-stable/center-unstable manifolds are in fact simply
stable/unstable.) In this section we consider the case when prior to
perturbation, for $\varepsilon=0$, the two manifolds $W_{\Lambda_{0}}^{s}$ and
$W_{\Lambda_{0}}^{u}$ are of the same dimension. Let us denote both these
dimensions by $k$ (meaning that $u=s=k$, $c=0$ and $n=c+u+s=2k$).

The most direct setting in which we can apply the approach from section
\ref{sec:zeros} to the detection of intersections of $W_{\Lambda_{\varepsilon
}}^{s}$ and $W_{\Lambda_{\varepsilon}}^{u}$ is when they are graphs over the
same domain. This is what is discussed in below motivating example:

\begin{example}
Assume that we have coordinates $\left(  x,y\right)  \in\mathbb{R}^{k}%
\times\mathbb{R}^{k}=\mathbb{R}^{n},$ such that%
\begin{equation}
\pi_{x}w^{u}\left(  \varepsilon,x\right)  =\pi_{x}w^{s}\left(  \varepsilon
,x\right)  =x. \label{eq:wu-ws-graphs-for-point}%
\end{equation}
We can define
\[
y\left(  \varepsilon,x\right)  :=\pi_{y}w^{u}\left(  \varepsilon,x\right)
-\pi_{y}w^{s}\left(  \varepsilon,x\right)  ,
\]
and apply Lemma \ref{lem:zero-from-degree} to establish that for any
$\varepsilon\in(0,\epsilon]$ we have intersections of $W_{\Lambda
_{\varepsilon}}^{u}$ with $W_{\Lambda_{\varepsilon}}^{s}$.
\end{example}

The setting in which the stable and unstable manifolds away from
$\Lambda_{\epsilon}$ would be given to us as graphs over the same coordinates
is rare. One needs some extra work to achieve this.

We assume that locally we have coordinate system $(x,y)\in\mathbb{R}^{k}%
\times\mathbb{R}^{k}$ such that for each $\varepsilon\in E$ the projection of
$W_{\lambda_{\varepsilon}}^{u,s}$ onto $x$ is a diffeomorphism, i.e. we have:

\begin{description}
\item[A] $\left[  \frac{\partial\pi_{x}w^{u}}{\partial\mathrm{u}}%
(E\times\overline{U^{\prime}})\right]  $ is an isomorphism, where $U^{\prime
}\subset\mathbb{R}^{k}$ is some bounded open set, and where $E=[0,\epsilon]$,

\item[B] $\left[  \frac{\partial\pi_{x}w^{s}}{\partial\mathrm{s}}%
(E\times\overline{S^{\prime}})\right]  $ is an isomorphism, where $S^{\prime
}\subset\mathbb{R}^{k}$ is some bounded open set.
\end{description}

Assume also that we have an open set $U\subset\mathbb{R}^{k}$, for which in the local coordinates
\begin{equation}
U\subset\left(  \pi_{x}w^{u}(E\times U^{\prime})\cap\pi_{x}w^{s}(E\times
S^{\prime})\right)  .\label{eq:common-projection-1}%
\end{equation}
Now we will change the parameterisation of manifolds $W_{\Lambda_{\varepsilon
}}^{u}$ and $W_{\Lambda_{\varepsilon}}^{s}$ so that they become graphs of
functions depending on $(\varepsilon,x)$. For this we assume that we can find
$C^{2}$ functions $\mathrm{u}:\mathbb{R\times R}^{k}\supset E\times
U\rightarrow\mathbb{R}^{k}$ and $\mathrm{s}:\mathbb{R\times R}^{k}\supset
E\times U\rightarrow\mathbb{R}^{k}$, for which for all $\varepsilon\in E$ and
$x\in U$%
\begin{equation}%
\begin{array}
[c]{r}%
\pi_{x}w^{u}\left(  \varepsilon,\mathrm{u}(\varepsilon,x)\right)
=x,\smallskip\\
\pi_{x}w^{s}\left(  \varepsilon,\mathrm{s}(\varepsilon,x)\right)  =x.
\end{array}
\label{eq:implicit-cond-fixed-pt}%
\end{equation}
Observe that conditions \textbf{A} and \textbf{B} together with
(\ref{eq:common-projection-1}) imply that functions $\mathrm{u}(\varepsilon
,x)$ and $\mathrm{s}(\varepsilon,x)$ exists and are as smooth as $w^{u}$ and
$w^{s}$.

Observe now that locally the manifolds $W_{\Lambda_{\varepsilon}}^{u}$ and
$W_{\Lambda_{\varepsilon}}^{s}$ have the expressions
\[
\{(x,\pi_{y}w^{u}(\varepsilon,\mathrm{u}(\varepsilon,x)):\varepsilon\in E,x\in
U\}\qquad\text{and}\qquad\{(x,\pi_{y}w^{s}(\varepsilon,\mathrm{s}%
(\varepsilon,x)):\varepsilon\in E,x\in U\},
\]
respectively. Let us define $y:\mathbb{R\times R}^{k}\supset E\times
U\rightarrow\mathbb{R}^{k}$ as
\begin{equation}
y\left(  \varepsilon,x\right)  :=\pi_{y}w^{u}\left(  \varepsilon
,\mathrm{u}(\varepsilon,x)\right)  -\pi_{y}w^{s}\left(  \varepsilon
,\mathrm{s}(\varepsilon,x)\right)  . \label{eq:f-for-fixed-pt}%
\end{equation}

\begin{theorem}
\label{th:fixed-pt}Assume that $\Lambda_{\varepsilon}=\left\{  \lambda
_{\varepsilon}\right\}  $ is a family of hyperbolic fixed points, that
$\mathrm{u}(\varepsilon,x)$ and $\mathrm{s}(\varepsilon,x)$, are functions for
which (\ref{eq:implicit-cond-fixed-pt}) holds true, and let $y\left(
\varepsilon,x\right)  $ be defined by (\ref{eq:f-for-fixed-pt}). If
assumptions of Lemma \ref{lem:zero-from-degree} are satisfied, then for any
$\varepsilon\in(0,\epsilon]$ the manifolds $W_{\Lambda_{\varepsilon}}^{u}$,
$W_{\Lambda_{\varepsilon}}^{s}$ intersect.
\end{theorem}

\begin{proof}
By Lemma \ref{lem:zero-from-degree}, for any $\varepsilon\in(0,\epsilon]$
there exists an $x\left(  \varepsilon\right)  $ such that $y(\varepsilon
,x\left(  \varepsilon\right)  )=0$. By the definition of $y$ in
(\ref{eq:f-for-fixed-pt}), if $y(\varepsilon,x\left(  \varepsilon\right)  )=0$
then 
\[\pi_{y}w^{u}\left(  \varepsilon,\mathrm{u}(\varepsilon,x\left(
\varepsilon\right)  )\right)  =\pi_{y}w^{s}\left(  \varepsilon,\mathrm{s}%
(\varepsilon,x\left(  \varepsilon\right)  )\right).\] From
(\ref{eq:implicit-cond-fixed-pt}) we also have%
\[
\pi_{x}w^{u}\left(  \varepsilon,\mathrm{u}(\varepsilon,x\left(  \varepsilon
\right)  )\right)  =x\left(  \varepsilon\right)  =\pi_{x}w^{s}\left(
\varepsilon,\mathrm{s}(\varepsilon,x\left(  \varepsilon\right)  )\right)  ,
\]
so we have an intersection point $w^{u}\left(  \varepsilon,\mathrm{u}%
(\varepsilon,x\left(  \varepsilon\right)  )\right)  =w^{s}\left(
\varepsilon,\mathrm{s}(\varepsilon,x\left(  \varepsilon\right)  )\right)  $,
as required.
\end{proof}

We note that in section \ref{sec:zeros}, in (\ref{eq:y1-y2-setting}), we have
two coordinates $y_{1}$ and $y_{2}$ for our $y$ defined in
(\ref{eq:f-for-fixed-pt}). In section \ref{sec:zeros}, in (\ref{eq:y2-zero}),
we assume that and on $y_{2}$ we have degenerate zeros. Depending on how the
manifolds $W_{\Lambda_{0}}^{u},$ $W_{\Lambda_{0}}^{s}$ intersect prior to
perturbation we will need to choose appropriate $k_{1}$ and $k_{2}$. For
instance, if the two manifolds coincide before the perturbation i.e.
$W_{\Lambda_{0}}^{u}=W_{\Lambda_{0}}^{s}$, then we should take $k_{1}=0$ and
$k_{2}=k.$ If on the other hand they intersect along a $l$-dimensional
manifold $W_{\Lambda_{0}}^{u}\cap W_{\Lambda_{0}}^{s}$, then we should take
$k_{1}=k-l$ and $k_{2}=l$. In such case we also need to carefully choose
coordinates $y_{1}$ and $y_{2}$ so that condition (\ref{eq:y2-zero}) is fulfilled.

To verify assumptions of Lemma \ref{lem:zero-from-degree} we can apply Lemma
\ref{lem:y-zero-practical}. We see that for this we need to be able to compute
first and second derivatives of $y\left(  \varepsilon,x\right)  $. This will
involve the computation of first and second derivatives of $\mathrm{u}%
(\varepsilon,x)$ and $\mathrm{s}(\varepsilon,x)$. These derivatives can be
computed implicitly as follows. By introducing notations
\begin{align}
g_{1},g_{2}  &  :\mathbb{R}\times\mathbb{R}^{k}\times\mathbb{R}^{k}%
\rightarrow\mathbb{R}^{k},\nonumber\\
g_{1}\left(  \varepsilon,x,\mathrm{u}\right)   &  :=\pi_{x}w^{u}\left(
\varepsilon,\mathrm{u}\right)  -x,\label{eq:dimensions-fixed-pt}\\
g_{2}\left(  \varepsilon,x,\mathrm{s}\right)   &  :=\pi_{x}w^{u}\left(
\varepsilon,\mathrm{s}\right)  -x,\nonumber
\end{align}
we see that the two equalities in (\ref{eq:implicit-cond-fixed-pt}) are
\begin{align}
g_{1}\left(  \varepsilon,x,\mathrm{u}\left(  \varepsilon,x\right)  \right)
&  =0,\label{eq:g1-zero}\\
g_{2}\left(  \varepsilon,x,\mathrm{s}\left(  \varepsilon,x\right)  \right)
&  =0. \label{eq:g2-zero}%
\end{align}
We see that (\ref{eq:g1-zero}) and (\ref{eq:g2-zero}) can be used for implicit
computations of the derivatives of $\mathrm{u}\left(  \varepsilon,x\right)  $
and $\mathrm{s}\left(  \varepsilon,x\right)  $, respectively. We discuss this
issue in section \ref{sec:implicit}, writing out all the details.

When we apply Lemma \ref{lem:y-zero-practical}, then we also obtain that the
intersection is transversal:

\begin{theorem}
Assume that $\Lambda_{\varepsilon}=\left\{  \lambda_{\varepsilon}\right\}  $
is a family of hyperbolic fixed points, that $\mathrm{u}(\varepsilon,x)$ and
$\mathrm{s}(\varepsilon,x)$, are functions for which
(\ref{eq:implicit-cond-fixed-pt}) holds true, and let $y\left(  \varepsilon
,x\right)  $ be defined by (\ref{eq:f-for-fixed-pt}). If assumptions of Lemma
\ref{lem:y-zero-practical} are satisfied, then for any $\varepsilon
\in(0,\epsilon]$ the manifolds $W_{\Lambda_{\varepsilon}}^{u}$, $W_{\Lambda
_{\varepsilon}}^{s}$ intersect transversally.
\end{theorem}

\begin{proof}
The result follows directly from Theorem \ref{th:fixed-pt} combined with
Lemmas~ \ref{lem:y-zero-practical}, \ref{lem:trans}.
\end{proof}

\subsection{NHIMs with stable/unstable manifolds of equal dimension}

In this section we consider $\Lambda_{\varepsilon}$, which are normally
hyperbolic invariant manifolds of dimension $c\neq0$. We assume that
$W_{\Lambda_{\varepsilon}}^{cu}$ and $W_{\Lambda_{\varepsilon}}^{cs}$ are
$C^{2},$ $k+c$ dimensional manifolds. The total dimension of our space is
$n=c+u+s=c+2k$.

As in section \ref{sec:fixed-pt} we start with a motivating example, in which
$W_{\Lambda_{\varepsilon}}^{cu}$ and $W_{\Lambda_{\varepsilon}}^{cs}$ are
graphs over some coordinates.

\begin{example}
Assume that we have coordinates $\left(  x,y,z\right)  \in\mathbb{R}^{k}%
\times\mathbb{R}^{k}\times\mathbb{R}^{c}$. We can think of $z$ as the center
coordinate associated with the manifolds $\Lambda_{\varepsilon}$, of $x$ as
the stable/unstable coordinate associated with the stable/unstable fibres of
the manifolds. The $y$ will be the coordinate along which we measure the
splitting. Assume that
\begin{align*}
\pi_{\left(  x,z\right)  }w^{cu}\left(  \varepsilon,x,z\right)   &  =\left(
x,z\right)  ,\\
\pi_{\left(  x,z\right)  }w^{cs}\left(  \varepsilon,x,z\right)   &  =\left(
x,z\right)  .
\end{align*}
We can fix some $z^{\ast}\in\mathbb{R}^{c}$ and define%
\[
y\left(  \varepsilon,x\right)  :=\pi_{y}w^{cu}\left(  \varepsilon,x,z^{\ast
}\right)  -\pi_{y}w^{cs}\left(  \varepsilon,x,z^{\ast}\right)  .
\]
Lemma \ref{lem:zero-from-degree} can now be employed to detect intersections
of $W_{\Lambda_{\varepsilon}}^{cu}$ with $W_{\Lambda_{\varepsilon}}^{cs}$ for
$\varepsilon\in(0,\epsilon].$
\end{example}

The setting in which the center-stable and center-unstable manifolds away from
$\Lambda_{\varepsilon}$ would be graphs over the same coordinates is rare. One
needs some extra work to achieve this.

Assume that there are local coordinates $\left(  x,y,z\right)  \in
\mathbb{R}^{k}\times\mathbb{R}^{k}\times\mathbb{R}^{c}$ such that
$W_{\Lambda_{\varepsilon}}^{cs}$ and $W_{\Lambda_{\varepsilon}}^{cu}$ project
locally in 1-1 way onto coordinates $(x,z)$ satisfying the following conditions:

\begin{description}
\item[A] $\left[  \frac{\partial\pi_{x,z}w^{cu}}{\partial\mathrm{cu}}(E\times
U^{\prime})\right]  $ is an isomorphism, where $U^{\prime}\subset
\mathbb{R}^{k}\times\mathbb{R}^{c}$ is some bounded open set, and $E=[0,\epsilon],$

\item[B] $\left[  \frac{\partial\pi_{x,z}w^{cs}}{\partial\mathrm{cs}}(E\times
S^{\prime})\right]  $ is an isomorphism, where $S^{\prime}\subset
\mathbb{R}^{k}\times\mathbb{R}^{c}$ is some bounded open set.
\end{description}

Assume also that we have open sets $U\subset\mathbb{R}^{k}, V\subset
\mathbb{R}^{c}$, for which
\[
U\times V\subset\left(  \pi_{x,z}w^{cu}(E\times U^{\prime})\cap\pi_{x,z}%
w^{cs}(E\times S^{\prime})\right)  .
\]
Now thanks to conditions \textbf{A} and \textbf{B} we can represent
$W_{\Lambda_{\varepsilon}}^{cs}$ and $W_{\Lambda_{\varepsilon}}^{cu}$ as
graphs of functions of $(\varepsilon,x,z)$. We therefore have $C^{2}$
functions
\begin{align*}
\mathrm{cu}  &  :\mathbb{R}\times\mathbb{R}^{k}\times\mathbb{R}^{c}\supset
E\times U\times V\rightarrow\mathbb{R}^{k}\times\mathbb{R}^{c},\\
\mathrm{cs}  &  :\mathbb{R}\times\mathbb{R}^{k}\times\mathbb{R}^{c}\supset
E\times U\times V\rightarrow\mathbb{R}^{k}\times\mathbb{R}^{c},
\end{align*}
for which for all $\left(  x,z\right)  \in U\times V$
\begin{equation}%
\begin{array}
[c]{l}%
\pi_{\left(  x,z\right)  }w^{cu}\left(  \varepsilon,\mathrm{cu}(\varepsilon
,x,z)\right)  =\left(  x,z\right)  ,\smallskip\\
\pi_{\left(  x,z\right)  }w^{cs}\left(  \varepsilon,\mathrm{cs}(\varepsilon
,x,z)\right)  =\left(  x,z\right)  .
\end{array}
\label{eq:implicit-lambda}%
\end{equation}

Note that locally the manifolds $W_{\Lambda_{\varepsilon}}^{cu}$,
$W_{\Lambda_{\varepsilon}}^{cs}$ have the expressions
\begin{align*}
\{(x,\pi_{y}w^{cu}(\varepsilon,\mathrm{cu}(\varepsilon,x,z),z)  &
:\varepsilon\in E,x\in U,z\in V\},\\
\{(x,\pi_{y}w^{cs}(\varepsilon,\mathrm{cs}(\varepsilon,x,z),z)  &
:\varepsilon\in E,x\in U,z\in V\},
\end{align*}
respectively.

Now we fix one $z^{\ast}\in V$ and define $y:\mathbb{R\times R}^{k}\supset
E\times U\rightarrow\mathbb{R}^{k}$ as
\begin{equation}
y\left(  \varepsilon,x\right)  :=\pi_{y}w^{cu}\left(  \varepsilon
,\mathrm{cu}(\varepsilon,x,z^{\ast})\right)  -\pi_{y}w^{cs}\left(
\varepsilon,\mathrm{cs}(\varepsilon,x,z^{\ast})\right)  . \label{eq:f-lambda}%
\end{equation}

\begin{theorem}
\label{th:nhim-same-dimension}Assume that functions $\mathrm{cu}%
(\varepsilon,x,z)$, $\mathrm{cs}(\varepsilon,x,z)$ satisfy
(\ref{eq:implicit-lambda}). If for fixed $z^{\ast}\in\mathbb{R}^{c}$ and for
$y\left(  \varepsilon,x\right)  $ defined by (\ref{eq:f-lambda}) the
assumptions of Lemma \ref{lem:zero-from-degree} are satisfied, then for any
$\varepsilon\in(0,\epsilon]$ the manifolds $W_{\Lambda_{\varepsilon}}^{cu}$,
$W_{\Lambda_{\varepsilon}}^{cs}$ intersect on the section $\{z=z^{\ast}\}$
\end{theorem}

\begin{proof}
The proof follows from mirror arguments to the proof of Theorem
\ref{th:fixed-pt}.
\end{proof}

Note that in section \ref{sec:zeros} we have a setting in which $y(\varepsilon
,x)=\left(  y_{1}(\varepsilon,x),y_{2}\left(  \varepsilon,x\right)  \right)  $
and $y_{2}\left(  0,x\right)  =0$ for all $x\in U$. The choice of the dimensions $k_{1},k_{2}$
of $y_{1}$ and $y_{2}$, respectively, depends on the initial setting of the
unperturbed manifolds. For instance, if the manifolds coincide, then we should
take $k_{1}=0$ and $k_{2}=k$. If the intersection $W_{\Lambda_{0}}^{cu}\cap
W_{\Lambda_{0}}^{cs}$ is a $c+l$ dimensional manifold, then we should choose
$k_{1}=k-l$ and $k_{2}=l$. We need also to ensure that the coordinates are
chosen accordingly so that $y_{2}\left(  0,x\right)  =0$.

The assumption of Lemma \ref{lem:zero-from-degree} for the $y(\varepsilon,x)$
defined in (\ref{eq:f-lambda}) can be verified using Lemma
\ref{lem:y-zero-practical}. In order to do so we need to compute the first and
second derivatives of $y\left(  \varepsilon,x\right)  $, which in turn
requires the computation of the first and second derivatives of $\mathrm{cu}%
(\varepsilon,x,z)$ and $\mathrm{cs}(\varepsilon,x,z)$. Below we discuss how
these can be computed implicitly. If we introduce the notations%
\[
g_{1},g_{2}:\mathbb{R}\times\mathbb{R}^{k+c}\times\mathbb{R}^{k+c}%
\rightarrow\mathbb{R}^{k+c},
\]%
\begin{align*}
g_{1}\left(  \varepsilon,\left(  x,z\right)  ,\mathrm{cu}\right)   &
=\pi_{\left(  x,z\right)  }w^{cu}\left(  \varepsilon,\mathrm{cu}\right)
-\left(  x,z\right)  ,\\
g_{2}\left(  \varepsilon,\left(  x,z\right)  ,\mathrm{cs}\right)   &
=\pi_{\left(  x,z\right)  }w^{cs}\left(  \varepsilon,\mathrm{cu}\right)
-\left(  x,z\right)  ,
\end{align*}
then the system of equations (\ref{eq:implicit-lambda}) is
\begin{align}
g_{1}\left(  \varepsilon,\left(  x,z\right)  ,\mathrm{cu}\left(
\varepsilon,x,z\right)  \right)   &  =0,\label{eq:implicit-g1}\\
g_{2}\left(  \varepsilon,\left(  x,z\right)  ,\mathrm{cs}\left(
\varepsilon,x,z\right)  \right)   &  =0. \label{eq:implicit-g2}%
\end{align}

The (\ref{eq:implicit-g1}--\ref{eq:implicit-g2}) can be used for the implicit
computations of the first and second derivatives of $\mathrm{cu}\left(
\varepsilon,x,z\right)  $ and $\mathrm{cs}\left(  \varepsilon,x,z\right)  $,
respectively. See section \ref{sec:implicit}.

When using Lemma \ref{lem:y-zero-practical} for establishing intersection of
the manifolds, we also obtain transversality (compare with Lemma~\ref{lem:trans}):

\begin{theorem}
\label{th:transversality-2}Assume that functions $\mathrm{cu}(\varepsilon
,x,z)$, $\mathrm{cs}(\varepsilon,x,z)$ satisfy (\ref{eq:implicit-lambda}). If
for $z^{\ast}\in\mathbb{R}^{c}$ for $y\left(  \varepsilon,x\right)  $ defined
by (\ref{eq:f-lambda}) the assumptions of Lemma \ref{lem:y-zero-practical} are
satisfied, then for any $\varepsilon\in(0,\epsilon]$ the manifolds
$W_{\Lambda_{\varepsilon}}^{cu}$, $W_{\Lambda_{\varepsilon}}^{cs}$ intersect transversally.
\end{theorem}

\begin{proof}
Let $\varepsilon\in(0,\epsilon]$ be fixed. By (\ref{eq:implicit-lambda}),
(\ref{eq:f-lambda}) and Lemma \ref{lem:y-zero-practical} combined with Theorem
\ref{th:nhim-same-dimension}, the manifolds $W_{\Lambda_{\varepsilon}}^{cu}$,
$W_{\Lambda_{\varepsilon}}^{cs}$ intersect at a point $p=p\left(
\varepsilon,z^{\ast}\right)  =w^{cu}\left(  \varepsilon,\mathrm{cu}%
(\varepsilon,x,z^{\ast})\right)  =w^{cs}\left(  \varepsilon,\mathrm{cs}%
(\varepsilon,x,z^{\ast})\right)  $, where $\left(  x,z^{\ast}\right)
=\pi_{\left(  x,z\right)  }p$. The coordinates in which we check
transversality are $\left(  x,y,z\right)  $. By (\ref{eq:implicit-lambda}) and
Lemma \ref{lem:trans} we have transversality in the $x,y$ coordinates:%
\begin{equation}
\left(  T_{p}W_{\Lambda_{\varepsilon}}^{cu}+T_{p}W_{\Lambda_{\varepsilon}%
}^{cs}\right)  \cap T_{p}\left\{  z=z^{\ast}\right\}  =\mathbb{R}^{2k}%
\times\{0\}^{c}. \label{eq:transversal-x-y}%
\end{equation}
Let $v\in\mathbb{R}^{c}$ be any given vector. By (\ref{eq:implicit-lambda}) it
follows that%
\begin{equation}
\frac{d}{ds}w^{cu}\left(  \varepsilon,\mathrm{cu}(\varepsilon,\pi_{x}%
p,z^{\ast}+sv)\right)  |_{s=0}=\left(  0,w,v\right)  \in T_{p}W_{\Lambda
_{\varepsilon}}^{cu}, \label{eq:z-span}%
\end{equation}
where
\[
w=\frac{d}{ds}\pi_{y}w^{cu}\left(  \varepsilon,\mathrm{cu}(\varepsilon,\pi
_{x}p,z^{\ast}+sv)\right)  |_{s=0}.
\]
Since we can take any $v\in\mathbb{R}^{c}$, from (\ref{eq:z-span}) we see that
$T_{p}W_{\Lambda_{\varepsilon}}^{cu}$ spans the direction of the coordinate
$z$. This combined with (\ref{eq:transversal-x-y}) gives%
\[
T_{p}W_{\Lambda_{\varepsilon}}^{cu}+T_{p}W_{\Lambda_{\varepsilon}}%
^{cs}=\mathbb{R}^{2k+c},
\]
meaning that $W_{\Lambda_{\varepsilon}}^{cu}$ and $W_{\Lambda_{\varepsilon}%
}^{cs}$ intersect transversally at $p$, as required.
\end{proof}

\subsection{Different dimensions of stable/unstable manifolds}

In this section we consider the case when $W_{\Lambda_{\varepsilon}}^{cu}$ and
$W_{\Lambda_{\varepsilon}}^{cs}$ have different dimensions. Here we assume
that $u<s$. (If we have opposite inequality, then we can swap the roles of the
stable/unstable manifolds in below setup.) As in previous sections, we start
with an example in which we have good coordinates for both $W_{\Lambda
_{\varepsilon}}^{cu}$ and $W_{\Lambda_{\varepsilon}}^{cs}$.

\begin{example}
Assume that we have coordinates
\[
\left(  x,y,v,z\right)  \in\mathbb{R}^{u}\times\mathbb{R}^{u}\times
\mathbb{R}^{s-u}\times\mathbb{R}^{c}=\mathbb{R}^{n}.
\]
We assume that $W_{\Lambda_{\varepsilon}}^{cu}$ is parameterised by $\left(
\varepsilon,x,z\right)  $ and that $W_{\Lambda_{\varepsilon}}^{cs}$ is
parameterised by $\left(  \varepsilon,x,v,z\right)  $. We assume that they are
$C^{2}$ graphs over these coordinates, meaning that%
\begin{align*}
\pi_{\left(  x,z\right)  }w^{cu}\left(  \varepsilon,x,z\right)   &  =\left(
x,z\right)  ,\\
\pi_{\left(  x,v,z\right)  }w^{cs}\left(  \varepsilon,x,v,z\right)   &
=\left(  x,v,z\right)  .
\end{align*}
Let $z^{\ast}\in\mathbb{R}^{c}$ be fixed and consider%
\[
y\left(  \varepsilon,x\right)  =\pi_{y}w^{cu}\left(  \varepsilon,x,z^{\ast
}\right)  -\pi_{y}w^{cs}\left(  \varepsilon,x,\pi_{v}w^{cu}\left(
\varepsilon,x,z^{\ast}\right)  ,z^{\ast}\right)  .
\]
Lemma \ref{lem:zero-from-degree} can now be employed to detect intersections
of $W_{\Lambda_{\varepsilon}}^{cu}$ with $W_{\Lambda_{\varepsilon}}^{cs}$ for
$\varepsilon\in(0,\epsilon].$
\end{example}

Once again, the setting in which the center-stable and center-unstable
manifolds away from $\Lambda_{\varepsilon}$ would be graphs over the same
coordinates is rare. We proceed similarly as in the previous sections.

We assume that there are local coordinates $\left(  x,y,v,z\right)
\in\mathbb{R}^{u}\times\mathbb{R}^{u}\times\mathbb{R}^{s-u}\times
\mathbb{R}^{c}$ such that $W_{\Lambda_{\varepsilon}}^{cu}$ and $W_{\Lambda
_{\varepsilon}}^{cs}$ project locally in 1-1 way onto coordinates $(x,z)$ and
$\left(  x,v,z\right)  $, respectively. This means that the following
conditions are satisfied:

\begin{description}
\item[A] $\left[  \frac{\partial\pi_{x,z}w^{cu}}{\partial\mathrm{cu}}(E\times
U^{\prime})\right]  $ is an isomorphism, where $U^{\prime}\subset
\mathbb{R}^{u}\times\mathbb{R}^{c}$ is some bounded open set, and $E=[0,\epsilon]$,

\item[B] $\left[  \frac{\partial\pi_{x,v,z}w^{cs}}{\partial\mathrm{cs}%
}(E\times S^{\prime})\right]  $ is an isomorphism, where $S^{\prime}%
\subset\mathbb{R}^{s}\times\mathbb{R}^{c}=\mathbb{R}^{u}\times\mathbb{R}%
^{s-u}\times\mathbb{R}^{c}$ is some bounded open set.
\end{description}

We also assume also that we have open sets $U \subset\mathbb{R}^{u}, V
\subset\mathbb{R}^{s-u}, K \subset\mathbb{R}^{c}$, for which
\begin{align*}
U\times K  &  \subset\pi_{x,z}w^{cu}(E\times U^{\prime}),\\
U\times V\times K  &  \subset\pi_{x,v,z}w^{cs}(E\times S^{\prime}),
\end{align*}
Now, thanks to conditions \textbf{A} and \textbf{B,} we can represent
$W_{\Lambda_{\varepsilon}}^{cs}$ and $W_{\Lambda_{\varepsilon}}^{cu}$ as
graphs of functions of $(\varepsilon,x,z)$. We therefore have $C^{2}$
functions
\begin{align*}
\mathrm{cu}  &  :\mathbb{R}\times\mathbb{R}^{u}\times\mathbb{R}^{c}\supset
E\times U\times K\rightarrow\mathbb{R}^{u}\times\mathbb{R}^{c},\\
\mathrm{cs}  &  :\mathbb{R}\times\mathbb{R}^{u}\times\mathbb{R}^{s-u}%
\times\mathbb{R}^{c}\supset E\times U\times V\times K\rightarrow\mathbb{R}%
^{u}\times\mathbb{R}^{c},
\end{align*}
for which
\begin{equation}
\pi_{\left(  x,z\right)  }w^{cu}\left(  \varepsilon,\mathrm{cu}\left(
\varepsilon,x,z\right)  \right)  =\left(  x,z\right)  ,
\label{eq:impl-different-dim-1}%
\end{equation}
and%
\begin{equation}
\pi_{\left(  x,v,z\right)  }w^{cs}\left(  \varepsilon,\mathrm{cs}\left(
\varepsilon,x,v,z\right)  \right)  =\left(  x,v,z\right)  .
\label{eq:impl-different-dim-2}%
\end{equation}

In the local coordinates, the manifolds $W_{\Lambda_{\varepsilon}}^{cs}$,
$W_{\Lambda_{\varepsilon}}^{cu}$ have the local expressions
\begin{align*}
&  \left\{  \left(  x,\pi_{\left(  y,v\right)  }w^{cu}\left(  \varepsilon
,\mathrm{cu}\left(  \varepsilon,x,z\right)  \right)  ,z\right)  :\varepsilon
\in E,x\in U,z\in K\right\}  ,\\
&  \left\{  \left(  x,\pi_{y}w^{cs}\left(  \varepsilon,\mathrm{cs}\left(
\varepsilon,x,v,z\right)  \right)  ,v,z\right)  :\varepsilon\in E,x\in U,v\in
V,z\in K\right\}  ,
\end{align*}
respectively.

We now fix $z^{\ast}\in\mathbb{R}^{c}$ and define $y:\mathbb{R\times R}%
^{u}\supset E\times U\rightarrow\mathbb{R}^{u}$ as%
\begin{equation}
y\left(  \varepsilon,x\right)  :=\pi_{y}w^{cu}\left(  \varepsilon
,\mathrm{cu}\left(  \varepsilon,x,z^{\ast}\right)  \right)  -\pi_{y}%
w^{cs}\left(  \varepsilon,\mathrm{cs}\left(  \varepsilon,x,\pi_{v}%
w^{cu}\left(  \varepsilon,\mathrm{cu}\left(  \varepsilon,x,z^{\ast}\right)
\right)  \right)  ,z^{\ast}\right)  . \label{eq:f-def-different-dim}%
\end{equation}

\begin{theorem}
\label{th:intersection-general}Assume that $u<s$. If $\mathrm{cs}\left(
\varepsilon,x,v,z\right)  ,$ $\mathrm{cu}\left(  \varepsilon,x,z\right)  $
satisfy (\ref{eq:impl-different-dim-1}--\ref{eq:impl-different-dim-2}) and for
$y\left(  \varepsilon,x\right)  $ defined in (\ref{eq:f-def-different-dim})
assumptions of Lemma \ref{lem:zero-from-degree} are satisfied, then for any
$\varepsilon\in(0,\epsilon]$ the manifolds $W_{\Lambda_{\varepsilon}}^{cu}$
and $W_{\Lambda_{\varepsilon}}^{cs}$ intersect.
\end{theorem}

\begin{proof}
The proof follows from mirror arguments to the proof of Theorem
\ref{th:fixed-pt}.
\end{proof}

Note that the domain and the range of $y(\varepsilon, \cdot)$ are in
$\mathbb{R}^{u}$. This means that we take $k=k_{1}+k_{2}=u$ in
(\ref{eq:y1-y2-setting}). Again, as in previous sections, the choice of the
dimensions $k_{1},k_{2}$ and of the coordinates $y_{1},y_{2}$ in
(\ref{eq:y1-y2-setting}) depends on how $W_{\Lambda_{0}}^{cu}$ and
$W_{\Lambda_{0}}^{cs}$ intersect.

Assumptions of Lemma \ref{lem:zero-from-degree} for $y(\varepsilon,x)$ defined
in (\ref{eq:f-def-different-dim}) can be verified using Lemma
\ref{lem:y-zero-practical}. In order to compute the first and second
derivatives of $\mathrm{cu}\left(  \varepsilon,x,z\right)  ,$ $\mathrm{cs}%
\left(  \varepsilon,x,v,z\right)  $ which are needed for the first and second
derivatives of $y\left(  \varepsilon,x\right)  $ we can do the following.
Consider
\begin{align*}
g_{1}  &  :\mathbb{R\times R}^{u+c}\mathbb{\times R}^{u+c}\rightarrow
\mathbb{R}^{u+c},\\
g_{1}\left(  \varepsilon,\left(  x,z\right)  ,\mathrm{cu}\right)   &
=\pi_{\left(  x,z\right)  }w^{cu}\left(  \varepsilon,\mathrm{cu}\right)
-\left(  x,z\right)  .
\end{align*}
The equation (\ref{eq:impl-different-dim-1}) in this notation is $g_{1}\left(
\varepsilon,\left(  x,z\right)  ,\mathrm{cu}\left(  \varepsilon,x,z\right)
\right)  =0$, which can be used for implicit computations of first and second
derivatives of $\mathrm{cu}\left(  \varepsilon,x,z\right)  $.

To compute first and second derivatives of $\mathrm{cs}\left(  \varepsilon
,x,v,z\right)  $ we can consider
\begin{align*}
g_{2}  &  :\mathbb{R\times R}^{s+c}\times\mathbb{R}^{s+c}\rightarrow
\mathbb{R}^{s+c}\\
g_{2}\left(  \varepsilon,\left(  x,v,z\right)  ,\mathrm{cs}\right)   &
=\pi_{\left(  x,v,z\right)  }w^{cs}\left(  \varepsilon,\mathrm{cs}\right)
-\left(  x,v,z\right)  .
\end{align*}
Then (\ref{eq:impl-different-dim-2}) is $g_{2}\left(  \varepsilon,\left(
x,v,z\right)  ,\mathrm{cs}\left(  \varepsilon,x,v,z\right)  \right)  =0$. This
can be used for implicit computations of first and second derivatives of
$\mathrm{cs}\left(  \varepsilon,x,v,z\right)  $. See section
\ref{sec:implicit}.

When we use Lemma \ref{lem:y-zero-practical}, we also obtain transversality:

\begin{theorem}
Assume that $u<s$. If $\mathrm{cs}\left(  \varepsilon,x,v,z\right)  ,$
$\mathrm{cu}\left(  \varepsilon,x,z\right)  $ satisfy
(\ref{eq:impl-different-dim-1}--\ref{eq:impl-different-dim-2}) and for
$y\left(  \varepsilon,x\right)  $ defined in (\ref{eq:f-def-different-dim})
assumptions of Lemma \ref{lem:y-zero-practical} are satisfied. Then for any
$\varepsilon\in(0,\epsilon]$ the manifolds $W_{\Lambda_{\varepsilon}}^{cu}$
and $W_{\Lambda_{\varepsilon}}^{cs}$ intersect transversally.
\end{theorem}

\begin{proof}
The proof follows from similar arguments to the proof of Theorem
\ref{th:transversality-2}.
\end{proof}


\section{Verification of assumptions \label{sec:verif}}

In this section we comment on how to obtain bounds on the function $y\left(
\varepsilon,x\right)  $ considered in (\ref{eq:f-for-fixed-pt}),
(\ref{eq:f-lambda}) and (\ref{eq:f-def-different-dim}), which are needed to
apply Lemma \ref{lem:y-zero-practical}. Our discussion will be divided into
two parts. The first concerns the bounds on the center-stable and
center-unstable manifolds. The second treats the implicit computations which
are needed for (\ref{eq:implicit-cond-fixed-pt}), (\ref{eq:implicit-lambda})
and (\ref{eq:impl-different-dim-1}--\ref{eq:impl-different-dim-2}).

\subsection{Obtaining bounds on the center-stable and center-unstable
manifolds}

In this section we discuss how one can obtain bounds on $w^{cu}(\varepsilon
,\mathrm{cu})$ and $w^{cs}(\varepsilon,\mathrm{cs})$. We are aware of two
alternative methods to obtain (local) bounds on the center-stable and
center-unstable manifolds of normally hyperbolic manifolds. The first is the
parametrisation method \cite{param-method, AlexJordi, llave4, llave5}, which
is a functional analytic method that allows for both efficient numerical
estimation of the invariant manifolds, as well as the for interval arithmetic
validation of bounds for the parameterisations of the manifolds together with
estimates on the derivatives. The second is a geometric method developed in
\cite{conecond, Geom, Meln}. This method can also be used for computer
assisted validation of the manifolds and the derivatives of their
parameterisations. In our example from section \ref{sec:example} we use the
second of the two methods.

Both of the above mentioned methods usually provide only local
parameterisations. Let us denote these as $w_{\mathrm{loc}}^{cu}%
:E\times\Lambda\times B_{u}\rightarrow\mathbb{R}^{n}\mathbb{\ }$%
and$\mathbb{\ }w_{\mathrm{loc}}^{cu}:E\times\Lambda\times B_{u}\rightarrow
\mathbb{R}^{n}$. These can be extended by using the flow $\Phi_{t}%
^{\varepsilon}$ of the ODE by fixing some $T>0$ and taking%
\begin{align*}
w^{cu}  &  =\Phi_{T}^{\varepsilon}\circ w_{\mathrm{loc}}^{cu},\\
w^{cs}  &  =\Phi_{-T}^{\varepsilon}\circ w_{\mathrm{loc}}^{cs}.
\end{align*}
There are efficient algorithms and packages that can be used to obtain
interval arithmetic enclosure of the flow, together with the (high order)
derivatives with respect to both the initial condition and parameter. In our
example from section \ref{sec:example} we use the CAPD\footnote{Computer
Assisted Proofs in Dynamics: http://capd.ii.uj.edu.pl} library, which is based
on the Lohner algorithms from \cite{WZ, Z}.

\subsection{Implicit computations\label{sec:implicit}}

The questions that interest us here is how to obtain bounds on $\mathrm{u}%
\left(  \varepsilon,x\right)  ,$ $\mathrm{s}\left(  \varepsilon,x\right)  $
from (\ref{eq:implicit-cond-fixed-pt}), and on $\mathrm{cu}$ and $\mathrm{cs}
$ from (\ref{eq:implicit-lambda}) and (\ref{eq:impl-different-dim-1}%
--\ref{eq:impl-different-dim-2}). We want to compute bounds on these
functions, together with their first and second derivatives.

All of the three cases: (\ref{eq:implicit-cond-fixed-pt}),
(\ref{eq:implicit-lambda}) and (\ref{eq:impl-different-dim-1}%
--\ref{eq:impl-different-dim-2}), can be treated in the same way. We will
describe this by introducing abstract notation. We will consider
$g:\mathbb{R}\times\mathbb{R}^{k}\times\mathbb{R}^{k}\rightarrow\mathbb{R}%
^{k}$ and will want to find $\kappa:\mathbb{R}\times\mathbb{R}^{k}%
\rightarrow\mathbb{R}^{k}$, together with its first and second derivatives,
such that%
\begin{equation}
g\left(  \varepsilon,x,\kappa\left(  \varepsilon,x\right)  \right)
=0.\label{eq:g-is-zero}%
\end{equation}
We assume that $\frac{\partial g}{\partial\kappa}$ is an isomorphism.

We start with the bounds for the image of $\kappa\left(  \varepsilon,x\right)
$. This can be obtained by using the Interval Newton method (see Theorem
\ref{th:interval-Newton}). To apply the theorem we can take a cubical set $K$
and $X=E\times U$, where $U$ is some cubical set in $\mathbb{R}^{k}$, and
verify that for some $k_{0}\in K$%
\begin{equation}
k_{0}-\left[  \frac{\partial g}{\partial\kappa}\left(  X,K\right)  \right]
^{-1}\left[  g\left(  X,k_{0}\right)  \right]  \subset\mathrm{int}K.
\label{eq:Newton-contraction}%
\end{equation}
If (\ref{eq:Newton-contraction}) is satisfied, then by Theorem
\ref{th:interval-Newton} we obtain that for any $\left(  \varepsilon,x\right)
\in E\times U$
\[
\kappa\left(  \varepsilon,x\right)  \in K.
\]

Now we turn to the derivatives of $\kappa$. By (\ref{eq:g-is-zero}),%
\begin{align*}
0  &  =\frac{d}{d\varepsilon}g\left(  \varepsilon,x,\kappa\left(
\varepsilon,x\right)  \right)  =\frac{\partial g}{\partial\varepsilon}\left(
\varepsilon,x,\kappa\left(  \varepsilon,x\right)  \right)  +\frac{\partial
g}{\partial\tau}\left(  \varepsilon,x,\kappa\left(  \varepsilon,x\right)
\right)  \frac{\partial\kappa}{\partial\varepsilon}\left(  \varepsilon
,x\right)  ,\\
0  &  =\frac{d}{dx}g\left(  \varepsilon,x,\kappa\left(  \varepsilon,x\right)
\right)  =\frac{\partial g}{\partial x}\left(  \varepsilon,x,\kappa\left(
\varepsilon,x\right)  \right)  +\frac{\partial g}{\partial\tau}\left(
\varepsilon,x,\kappa\left(  \varepsilon,x\right)  \right)  \frac
{\partial\kappa}{\partial x}\left(  \varepsilon,x\right)  .
\end{align*}
This means that for $K$ and $X$ for which (\ref{eq:Newton-contraction}) holds,
for $\left(  \varepsilon,x\right)  \in E\times U$%
\begin{align}
\frac{\partial\kappa}{\partial\varepsilon}\left(  \varepsilon,x\right)   &
\in-\left[  \frac{\partial g}{\partial\kappa}\left(  \varepsilon,x,K\right)
\right]  ^{-1}\left[  \frac{\partial g}{\partial\varepsilon}\left(
\varepsilon,x,K\right)  \right]  ,\label{eq:kappa-e-enclosure}\\
\frac{\partial\kappa}{\partial x}\left(  \varepsilon,x\right)   &  \in-\left[
\frac{\partial g}{\partial\kappa}\left(  \varepsilon,x,K\right)  \right]
^{-1}\left[  \frac{\partial g}{\partial x}\left(  \varepsilon,x,K\right)
\right]  . \label{eq:kappa-x-enclosure}%
\end{align}
Also from (\ref{eq:g-is-zero}),
\begin{align*}
0  &  =\frac{d^{2}g}{d\varepsilon dx}\\
&  =\frac{d}{d\varepsilon}\left(  \frac{\partial g}{\partial x}+\frac{\partial
g}{\partial\kappa}\frac{\partial\kappa}{\partial x}\right) \\
&  =\frac{\partial^{2}g}{\partial\varepsilon\partial x}+\frac{\partial^{2}%
g}{\partial\kappa\partial x}\frac{\partial\kappa}{\partial\varepsilon}+\left(
\frac{\partial^{2}g}{\partial\varepsilon\partial\kappa}+\frac{\partial^{2}%
g}{\partial\kappa^{2}}\frac{\partial\kappa}{\partial\varepsilon}\right)
\frac{\partial\kappa}{\partial x}+\frac{\partial g}{\partial\kappa}%
\frac{\partial^{2}\kappa}{\partial\varepsilon\partial x}%
\end{align*}
so%
\begin{equation}
\frac{\partial^{2}\kappa}{\partial\varepsilon\partial x}=-\left(
\frac{\partial g}{\partial\kappa}\right)  ^{-1}\left(  \frac{\partial^{2}%
g}{\partial\varepsilon\partial x}+\frac{\partial^{2}g}{\partial\kappa\partial
x}\frac{\partial\kappa}{\partial\varepsilon}+\left(  \frac{\partial^{2}%
g}{\partial\varepsilon\partial\kappa}+\frac{\partial^{2}g}{\partial\kappa^{2}%
}\frac{\partial\kappa}{\partial\varepsilon}\right)  \frac{\partial\kappa
}{\partial x}\right)  . \label{eq:kappa-e-x}%
\end{equation}
Above can be used to obtain an interval enclosure in the same way as
(\ref{eq:kappa-e-enclosure}--\ref{eq:kappa-x-enclosure}). In
(\ref{eq:kappa-e-x}) we can use the enclosures (\ref{eq:kappa-e-enclosure}%
--\ref{eq:kappa-x-enclosure}) for the derivatives $\frac{\partial\kappa
}{\partial\varepsilon}\left(  \varepsilon,x\right)  $ and $\frac
{\partial\kappa}{\partial x}\left(  \varepsilon,x\right)  $.

\begin{remark}
If we consider $g$ that is defined by (\ref{eq:implicit-cond-fixed-pt}) or
(\ref{eq:implicit-lambda}), then (\ref{eq:kappa-e-x}) simplifies since
$\frac{\partial^{2}g}{\partial\varepsilon\partial x}=0$ and $\frac
{\partial^{2}g}{\partial\kappa\partial x}=0.$
\end{remark}


\section{Example of application \label{sec:example}}

We consider an example from \cite{Wiggins}, which was introduced by Lerman and
Umanski\u{\i} \cite{Lerman}:%
\begin{equation}
x^{\prime}=F\left(  \varepsilon,x\right)  ,\label{eq:our-ode}%
\end{equation}
where the vector field $F:\mathbb{R}\times\mathbb{R}^{4}\rightarrow
\mathbb{R}^{4}$ is%
\[
F=\left(
\begin{array}
[c]{rrrr}%
\lambda & -\omega & 0 & 0\\
\omega & \lambda & 0 & 0\\
0 & 0 & -\lambda & -\omega\\
0 & 0 & \omega & -\lambda
\end{array}
\right)  \left(
\begin{array}
[c]{c}%
x_{1}\\
x_{2}\\
x_{3}\\
x_{4}%
\end{array}
\right)  +\left(
\begin{array}
[c]{c}%
-2\sqrt{2}\lambda x_{3}\left(  x_{3}^{2}+x_{4}^{2}\right)  \\
-2\sqrt{2}\lambda x_{4}\left(  x_{3}^{2}+x_{4}^{2}\right)  \\
2\sqrt{2}\lambda x_{1}\left(  x_{1}^{2}+x_{2}^{2}\right)  \\
2\sqrt{2}\lambda x_{2}\left(  x_{1}^{2}+x_{2}^{2}\right)
\end{array}
\right)  +\varepsilon\left(
\begin{array}
[c]{c}%
x_{2}\\
0\\
x_{4}\\
0
\end{array}
\right)  .
\]
We consider $\lambda=\omega=1$. The zero is a fixed point with a two
dimensional unstable and a two dimensional stable manifold. For $\varepsilon
=0$, the unstable manifold coincides
with the stable manifold (see Figure \ref{fig:homoclinic}).

We shall apply our method to prove the following:

\begin{figure}[ptb]
\begin{center}
\includegraphics[height=5cm]{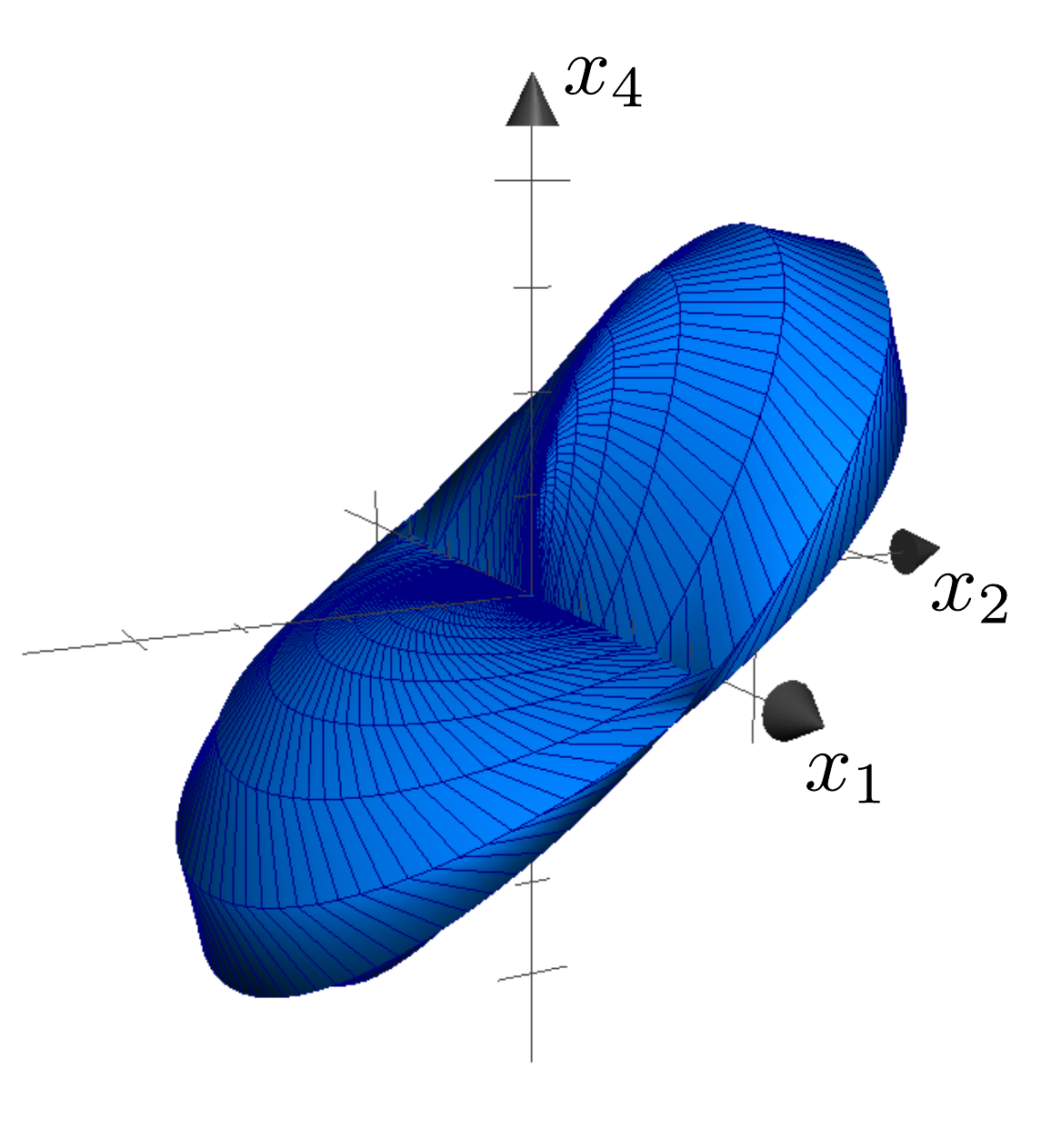}
\end{center}
\caption{The two dimensional stable manifold of zero coincides with the two
dimensional stable manifold. The plot depicts their projection onto coordinates
$x_{1},x_{2}$ and $x_{4}$.\label{fig:homoclinic}}%
\end{figure}

\begin{theorem}
\label{th:example}For any $\varepsilon\in (0,10^{-7}]$ the stable and
unstable manifolds of the origin intersect transversally within a $10^{-5}$
distance of the point $p_{0}=\left(  2^{-1/4},0,2^{-1/4}%
,0\right)  .$
\end{theorem}

In subsequent sections we go over the steps which were taken to prove the result.

\subsection{Local bounds on the stable/unstable
manifolds\label{sec:local-bounds}}

For (\ref{eq:our-ode}) with $\varepsilon=0$ it is possible to derive analytic formulae for the parameterisation of the stable/unstable manifold (see \cite{Wiggins}). We do not make use of this. We choose not to, since our method does not require knowing the analytic formulae for the separatrices. For us it is sufficient to establish bounds on the manifolds, which is what we discuss in this section.

To obtain computer assisted bounds on the manifolds we use the method developed in \cite{Meln}. It
can be used to establish bounds on stable and unstable manifolds, together with their derivatives, within a
given neighbourhood of a fixed point. It is best o obtain such bounds in 
local coordinates in which the invariant manifold is `straightened out'. We start by discussing a change of coordinates
that we used to achieve this.

For $\varepsilon=0$ the system is generated by a hamiltonian $H$ with an integral $K$
of the form%
\begin{align*}
H &  =\lambda\left(  x_{1}x_{3}+x_{2}x_{4}\right)  -\omega\left(  x_{2}%
x_{3}-x_{1}x_{4}\right)  -\frac{\lambda}{\sqrt{2}}\left[  \left(  x_{1}%
^{2}+x_{2}^{2}\right)  ^{2}+\left(  x_{3}^{2}+x_{4}^{2}\right)  ^{2}\right]
,\\
K &  =x_{2}x_{3}-x_{1}x_{4},
\end{align*}
where the $(x_1,x_2)$ are the positions and $(x_3,x_4)$ are their conjugated momenta.
We can use the integrals $H,K$ to approximate the two dimensional
stable/unstable manifolds. We discuss this for the unstable manifold, which is
tangent to the vector space spanned on the coordinates $x_{1}$ and $x_{2}$.
From this tangency, for small $\left\Vert x\right\Vert $ the $x_{3}$ and
$x_{4}$ along the unstable manifold will be small. The fixed point has
$H=K=0$, meaning that this will be preserved along the manifold. Since $K=0$,
and $\left(  x_{3}^{2}+x_{4}^{2}\right)  ^{2}$ is small in comparison to the
remaining terms of $H$, from $H=0$ we obtain
\[
\lambda\left(  x_{1}x_{3}+x_{2}x_{4}\right)  -\lambda2^{-1/2}\left(  x_{1}%
^{2}+x_{2}^{2}\right)  ^{2}\approx0.
\]
Thus, the unstable manifold is approximated by points satisfying%
\begin{align*}
x_{2}x_{3}-x_{1}x_{4} &  =0,\\
x_{1}x_{3}+x_{2}x_{4} &  =2^{-1/2}\left(  x_{1}^{2}+x_{2}^{2}\right)  ^{2}.
\end{align*}
For fixed $x_{1},x_{2}$, above is a linear equation for $x_{3},x_{4}$, with
the solution $x_{3}=2^{-1/2}x_{1}\left(  x_{1}^{2}+x_{2}^{2}\right)  ,$
$x_{4}=2^{-1/2}x_{2}\left(  x_{1}^{2}+x_{2}^{2}\right)  ^{2}.$ This means that
close to the origin the unstable manifold can be approximated by
\[
W^{u}\approx\left\{  \left(  x_{1},x_{2},2^{-1/2}x_{1}\left(  x_{1}^{2}%
+x_{2}^{2}\right)  ,2^{-1/2}x_{2}\left(  x_{1}^{2}+x_{2}^{2}\right)
^{2}\right)  :x_{1},x_{2}\text{ are small}\right\}  .
\]

From mirror arguments it follows that
\[
W^{s}\approx\left\{  \left(  2^{-1/2}x_{3}\left(  x_{3}^{2}+x_{4}^{2}\right)
,2^{-1/2}x_{2}\left(  x_{3}^{2}+x_{4}^{2}\right)  ^{2},x_{3},x_{4}\right)
:x_{3},x_{4}\text{ are small}\right\}  .
\]

We can therefore use the following change of coordinates that straighten out
the unstable manifold. Let $v=\left(  v_{1},\ldots,v_{4}\right)  $ denote the
local coordinates defined by
\begin{equation}
x=\Psi_{u}\left(  v\right)  =\left(  v_{1},v_{2},v_{3}+2^{-1/2}\left(
v_{1}^{2}+v_{2}^{2}\right)  v_{1},v_{4}+2^{-1/2}\left(  v_{1}^{2}+v_{2}%
^{2}\right)  v_{2}\right)  .\label{eq:local-unst}%
\end{equation}
In the local coordinates $\left\{  v_{3}=v_{4}=0\right\}  $ approximates the
unstable manifold. 

To straighten out
the stable manifold we use local coordinates given by%
\[
x=\Psi_{s}\left(  v\right)  =\left(  v_{1}+2^{-1/2}v_{3}\left(  v_{3}%
^{2}+v_{4}^{2}\right)  ,v_{2}+2^{-1/2}v_{2}\left(  v_{3}^{2}+v_{4}^{2}\right)
^{2},v_{3},v_{4}\right).
\]
In these coordinates $\left\{  v_{1}=v_{2}=0\right\}  $ approximates the
stable manifold.

We now discuss how we obtain local bounds on the unstable manifold in the local coordinates given by $\Psi_u$.
The inverse change to $\Psi_u$ is%
\[
v=\Psi_{u}^{-1}\left(  x\right)  =\left(  x_{1},x_{2},x_{3}-2^{-1/2}\left(
x_{1}^{2}+x_{2}^{2}\right)  x_{1},x_{4}-2^{-1/2}\left(  x_{1}^{2}+x_{2}%
^{2}\right)  x_{2}\right)  .
\]
The formula for the vector field in the local coordinates
\[
v^{\prime}=G\left(  \varepsilon,v\right)  =D\Psi_{u}^{-1}\left(  \Psi_{u}\left(
v\right)  \right)  F(\varepsilon,\Psi_{u}\left(  v\right)  ),
\]
can be easily derived, though it is somewhat lengthy so we do not write it out
here. We use this vector field to establish the bounds on the unstable
manifold using the method described in \cite[Theorems 30, 36]{Meln}. In our
application we extend the phase space to include the parameter $\varepsilon$%
\[
\left(  \varepsilon,v\right)  ^{\prime}=\left(  0,G\left(  \varepsilon
,v\right)  \right)  .
\]
This means that we consider a normally hyperbolic manifold $\Lambda=\left\{
\left(  \varepsilon,0\right)  :\varepsilon \in E\right\}  $, where 
\[E=\left[0,10^{-7}\right]\] 
is an interval of
parameters. We use the same $\Psi_{u}$ for all $\varepsilon\in E  $. The method from \cite{Meln} allowed us to obtain
the bounds on the manifold in the extended phase space, together with the
bounds on its first and second derivatives.
The obtained local bounds on the unstable manifold are in the form of a graph%
\begin{equation}
\left\{  v=\left(  \varepsilon,v_{1},v_{2},w_{\mathrm{loc}}^{u}\left(
\varepsilon,v_{1},v_{2}\right)  \right)  :\varepsilon,v_{1},v_{2}\in
B_{u}\right\}  ,\label{eq:unst-bd}
\end{equation}
where $B_{u}=E\times\left[  -r,r\right]  \times\left[  -r,r\right]  $, with
$r=1.5\cdot10^{-4}$. We prove that $w_{\mathrm{loc}}^{u}:B_{u}\rightarrow
\mathbb{R}^{2}$ is Lipschitz with the constant $L=10^{-8}.$ In Figure
\ref{fig:WuLocal} we give the plot of the obtained bound. By looking at the
scale on the vertical axis, we see that the bounds are quite sharp. Moreover, from the
method we establish that
\[
\left\vert \frac{\partial^{2}\pi_{v_{3}}w_{\mathrm{loc}}^{u}}{\partial
v_{i}\partial v_{j}}\right\vert ,\left\vert \frac{\partial^{2}\pi_{v_{3}%
}w_{\mathrm{loc}}^{u}}{\partial\varepsilon\partial v_{j}}\right\vert
\leq3.518\cdot10^{-5}\qquad\text{for }i,j=1,2.
\]

The bounds can easily be transported to the original coordinates
$x_{1},x_{2},x_{3},x_{4}$ using (\ref{eq:local-unst}). In Figure
\ref{fig:Wu} we give a plot in of the bounds from Figure \ref{fig:WuLocal}
transported to the original coordinates of the system. 

We use the same method to establish bounds on the stable manifolds an its
derivatives. The Lipschitz bound on the slope of the manifold and the bounds on the second
derivatives are identical as for the unstable manifold.

\begin{figure}[ptb]
\begin{center}
\includegraphics[height=4cm]{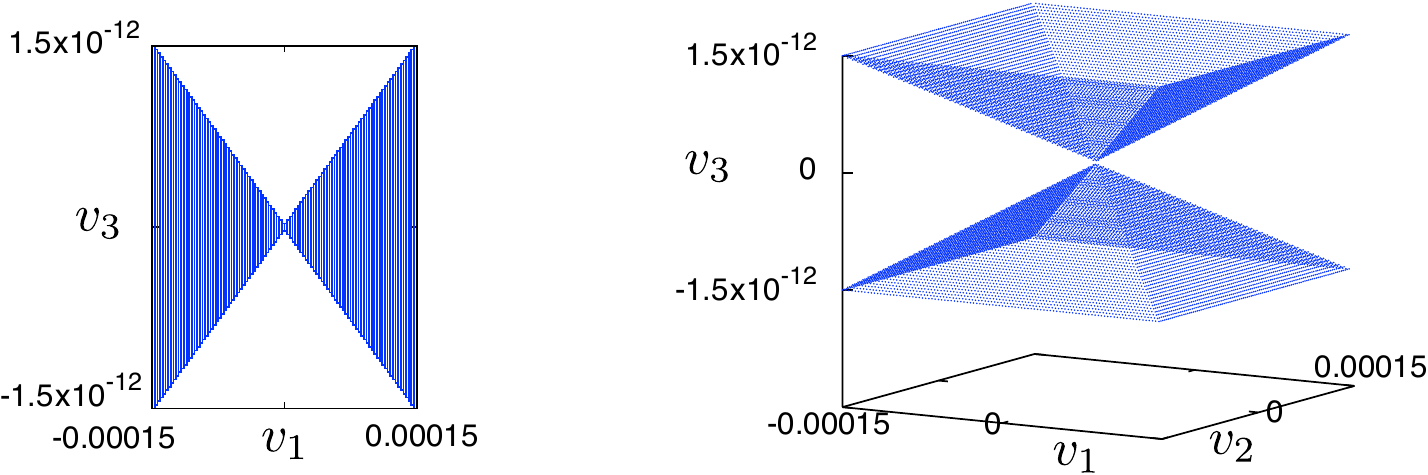}
\end{center}
\caption{Projections of the bounds on the unstable manifold in local
coordinates. On the left, we have the intersection of the bound on the manifold (\ref{eq:unst-bd}) with
$\{v_{2}=0\}$, projected onto $v_{1},v_{3}$ coordinates. The manifold lies
within the boxed area. On the right we have the projection of the bound on the
manifold (\ref{eq:unst-bd}) onto the $v_{1},v_{2},v_{3}$ coordinates. The manifold is a horizontal
two dimensional surface, which lies between the upper and the lower
cones. }%
\label{fig:WuLocal}%
\end{figure}\begin{figure}[ptb]
\begin{center}
\includegraphics[height=4cm]{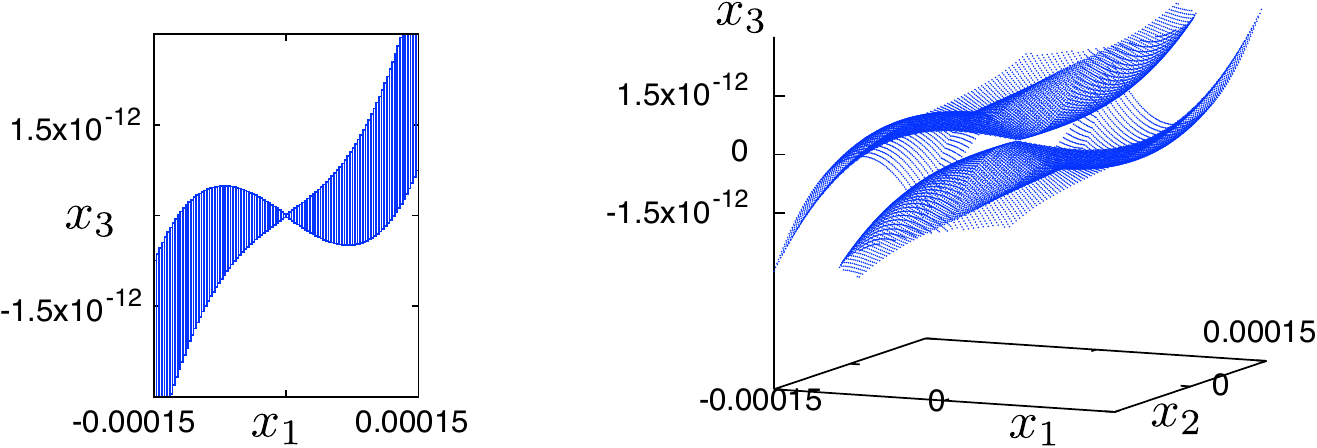}
\end{center}
\caption{The bounds from Figure \ref{fig:WuLocal} in the original
coordinates.}%
\label{fig:Wu}%
\end{figure}

\subsection{The distance function at the intersection point}

We will look for an intersection of the invariant manifolds close to the point $p_{0}=\left(
2^{-1/4},0,2^{-1/4},0\right)  $. (We have found this point by considering
$H=0$ and taking $x_{2}=x_{4}=0$.) The stable and unstable manifolds are tangent to
$\left\{  H=0\right\}  $ and $\left\{  K=0\right\}  $. We compute two tangent
vectors%
\begin{equation}
J\nabla H\left(  p_{0}\right)  =2^{-1/4}\left(
\begin{array}
[c]{c}%
-\lambda\\
\omega\\
\lambda\\
\omega
\end{array}
\right)  ,\qquad\qquad J\nabla K\left(  p_{0}\right)  =2^{-1/4}\left(
\begin{array}
[c]{c}%
0\\
-1\\
0\\
-1
\end{array}
\right)  ,\label{eq:H-K-tangents}%
\end{equation}
where
\[
J=\left(
\begin{array}
[c]{ll}%
0 & \mathrm{Id}\\
-\mathrm{Id} & 0
\end{array}
\right),  \qquad\mathrm{Id}=\left(
\begin{array}
[c]{ll}%
1 & 0\\
0 & 1
\end{array}
\right)  .
\]
For $\left(  \mathrm{x},\mathrm{y}\right)  =\left(  \mathrm{x}_{1}%
,\mathrm{x}_{2},\mathrm{y}_{1},\mathrm{y}_{2}\right)  \in\mathbb{R}^{4},$ we
choose the local coordinates around $p_{0}$ defined as
\begin{equation}
x=V\left(  \mathrm{x},\mathrm{y}\right)  :=p_{0}+A\left(
\begin{array}
[c]{c}%
\mathrm{x}\\
\mathrm{y}%
\end{array}
\right),  \qquad\text{for\qquad}A=\left(
\begin{array}
[c]{rrrr}%
-\lambda & 0 & 1 & 0\\
\omega & -1 & 0 & 1\\
\lambda & 0 & 1 & 0\\
\omega & -1 & 0 & -1
\end{array}
\right)  .\label{eq:V-def}%
\end{equation}
The first two columns in $A$ are taken based on (\ref{eq:H-K-tangents}). The
last two are chosen so that the two vectors composed of columns three and four
would be orthogonal to the others.

We can now propagate the bounds on the stable and unstable manifolds from
section \ref{sec:local-bounds} to the intersection point as follows. We 
take a fixed $T>0$ and define%
\begin{align*}
w^{u}\left(  \varepsilon,x\right)    & :=V^{-1}\circ\Phi_{T}^{\varepsilon
}\circ\Psi_{u}\left(  x,w_{\mathrm{loc}}^{u}\left(  \varepsilon,x\right)
\right)  ,\\
w^{s}\left(  \varepsilon,x\right)    & :=V^{-1}\circ\Phi_{-T}^{\varepsilon
}\circ\Psi_{s}\left(  w_{\mathrm{loc}}^{s}\left(  \varepsilon,x\right)
,x\right)  .
\end{align*}
In our computer assisted proof we take $T=9,$ which is sufficient to reach
$p_{0}$ from the local bounds established in section \ref{sec:local-bounds}.
\begin{remark} We are making our computations using the CAPD\footnote{Computer
Assisted Proofs in Dynamics: http://capd.ii.uj.edu.pl/} library. The library performs rigorous propagation of interval enclosures of jets along the flow of a vector field. This allows us to propagate in interval arithmetic the local bounds obtained in section \ref{sec:local-bounds}.
\end{remark}
We then use the method outlined in section \ref{sec:implicit} to establish bounds on $u\left(
\varepsilon,\mathrm{x}\right)  $ and $s\left(  \varepsilon,\mathrm{x}\right)
$ which solve
\begin{align*}
\pi_{\mathrm{x}}w^{u}\left(  \varepsilon,u\left(  \varepsilon,\mathrm{x}%
\right)  \right)  -\mathrm{x}  & =0,\\
\pi_{\mathrm{x}}w^{s}\left(  \varepsilon,s\left(  \varepsilon,\mathrm{x}%
\right)  \right)  -\mathrm{x}  & =0.
\end{align*}
The distance function which we consider for
the proof of the transversal intersection is
\begin{equation}
\mathrm{y}\left(  \varepsilon,\mathrm{x}\right)  :=\pi_{\mathrm{y}}%
w^{u}\left(  \varepsilon,u\left(  \varepsilon,\mathrm{x}\right)  \right)
-\pi_{\mathrm{y}}w^{s}\left(  \varepsilon,s\left(  \varepsilon,\mathrm{x}%
\right)  \right)  .\label{eq:y-example}%
\end{equation}

\subsection{Computer assisted bounds}

In this section we take $\mathrm{y}\left(  \varepsilon,\mathrm{x}\right)  $
defined in (\ref{eq:y-example}) and use Lemmas \ref{lem:y-zero-practical},
\ref{lem:trans} to establish the proof of Theorem \ref{th:example}.

\begin{remark}
Since the manifolds are both two dimensional and prior to the perturbation
they overlap we take $k_{1}=0$ and $k_{2}=2$ in Lemmas
\ref{lem:y-zero-practical}, \ref{lem:trans}. This means that there are no $A_{11}, A_{12}, A_{21}$ matrices to consider and we do not need to check the condition (\ref{eq:estmA11}).
\end{remark}

The point $p_{0}=\left(  2^{-1/4},0,2^{-1/4},0\right)  $ when transported to
the local coordinates given by $V$ defined in (\ref{eq:V-def}) is the origin.
This means that we are looking for zeros of $\mathrm{y}\left(  \varepsilon
,\mathrm{x}\right)  $ in a neighbourhood of zero.

We consider $U:=E\times\left[  -R,R\right]  \times\left[  -R,R\right]  $ with
$R=10^{-5}.$ Using computer assisted computations in CAPD we have obtained the
following bounds:%
\begin{align}
A_{22} & =\left[  \frac{\partial^{2}\mathrm{y}}{\partial\varepsilon
\partial\mathrm{x}}\left(  \varepsilon=0,\mathrm{x}=0\right)  \right]
\label{eq:A-bound}\\
& =\left(
\begin{array}
[c]{ll}%
\lbrack5.878219435,5.878219454] & \left[  -13.12140618,-13.12140616\right]  \\
\left[  4.972558758,4.97255877\right]   & \left[
-2.358981737,-2.358981727\right]
\end{array}
\right)  ,\nonumber
\end{align}%
\begin{align}
\Delta_{2}  & =\left[  \frac{\partial^{2}\mathrm{y}}{\partial\varepsilon
\partial\mathrm{x}}\left(  U\right) -A_{22} \right]  \label{eq:Delta-bound}\\
& =\left(
\begin{array}
[c]{ll}%
\left[  -1.299703331,1.286153144\right]   & \left[
-0.9977804236,0.9891960037\right]  \\
\left[  -0.7568318161,0.7534173913\right]   & \left[
-0.5842185843,0.5818916067\right]
\end{array}
\right)  ,\nonumber
\end{align}
and%
\begin{equation}
\left[  \frac{\partial\mathrm{y}}{\partial\varepsilon}\left(  U\right)
\right]  =\left(
\begin{array}
[c]{c}%
\left[  -1.030549066e-05,1.030549066e-05\right]  \\
\left[  -9.608989689e-06,9.608989695e-06\right]
\end{array}
\right)  .\label{eq:y-bound}%
\end{equation}

From (\ref{eq:A-bound}--\ref{eq:y-bound}) we compute that
\begin{align*}
m\left(  A_{22}\right)    & \geq3.423087786,\\
\left\Vert \Delta_{2}\right\Vert  & \leq2.000249209,\\
\left\Vert \frac{\partial\mathrm{y}}{\partial\varepsilon}\left(  U\right)
\right\Vert  & \leq1.409027398\cdot10^{-5},
\end{align*}
which gives
\[
m\left(  A_{22}\right)  R-\left\Vert \frac{\partial\mathrm{y}}{\partial\varepsilon
}\left(  U\right)  \right\Vert -\left\Vert \Delta_{2}\right\Vert
R>1.38\cdot10^{-7}>0.
\]
This by Lemmas \ref{lem:y-zero-practical}, \ref{lem:trans} establishes the
proof of Theorem \ref{th:example}.

The computations needed for the proof of Theorem \ref{th:example} have taken under a minute on a single core 3 GHz Intel Core i7 processor.

We have considered perturbations $\varepsilon \in(0,10^{-7}]$, but the proof of intersection can easily be extended to larger $\varepsilon$.


\begin{thebibliography}{00}


\bibitem{BJLM} J.B. van den Berg, J.D. Mireles-James, J.-P. Lessard and K. Mischaikow. {\em Rigorous numerics for symmetric connecting orbits: even homoclinics of the Gray-Scott equation}. SIAM J. Math. Anal. 43 (2011), no. 4, 1557--1594.

\bibitem{A} G. Alefeld, Inclusion methods for systems of nonlinear
equations - the interval Newton method and modifications, in
\emph{ Topics in Validated Computations},
   J. Herzberger (Editor), Elsevier Science B.V., 1994,
 pages 7--26

\bibitem{param-method} X. Cabr\'e, E. Fontich, R. de la Llave,  \emph{The parameterization method for invariant manifolds III: overview and applications}, J. Diff. Eq., 218 (2005) 444--515.

\bibitem{Chow} S.N. Chow, J.K. Hale, J. Mallet-Paret, \emph{An example of bifurcation to homoclinic orbits}, J. Differential Equations 37 (1980) 351?373.

\bibitem{delshamsMelPotential} A. Delshams, R. Ramirez-Ros, \emph{ Melnikov Potential for Exact Symplectic Maps}, Commun. Math. Phys. 190, 213 – 245 (1997)

\bibitem{delshamsLlave} A. Delshams, R. de la Llave, T.M. Seara, \emph{Geometric properties of the scattering map of a normally hyperbolic invariant manifold}, Adv. Math. 217 (3) (2008) 1096--1153.

\bibitem{delshamsGuttSpilt2000}  A. Delshams, P. Gutierrez, \emph{Splitting Potential and the Poincare-Melnikov Method
for Whiskered Tori in Hamiltonian Systems}, J. Nonlinear Sci.  10, 433–476 (2000)

\bibitem{AlexJordi} A. Haro, M. Canadell, J.-L. Figueras, A. Luque, Alejandro, J.-M. Mondelo, \emph{The parameterization method for invariant manifolds. From rigorous results to effective computations}. Applied Mathematical Sciences, 195. Springer, 2016.

\bibitem{llave4}  A. Haro, R. de la Llave, \emph{A parameterization method for the computation of invariant tori and their whiskers in quasi-periodic maps: numerical algorithms}. Discrete Contin. Dyn. Syst. Ser. B 6 (2006), no. 6, 1261--1300.

\bibitem{llave5} A. Haro, R. de la Llave, \emph{A parameterization method for the computation of invariant tori and their whiskers in quasi-periodic maps: rigorous results.} J. Differential Equations 228 (2006), no. 2, 530--579.




\bibitem{Jay} M.J. Capinski, J.D. Mireles James, \emph{Validated computation of heteroclinic sets}, SIAM J. Appl. Dyn. Syst. Vol. 16, No. 1, (2017) pp. 375--409

\bibitem{conecond} M.J. Capi\'nski, P. Zgliczy\'nski, \emph{Cone conditions and covering relations for topologically normally hyperbolic manifolds}, Discrete Contin. Dyn. Syst. 30 (2011) 641--670.

\bibitem{Geom} M.J. Capi\'nski, P. Zgliczy\'nski, \emph{Geometric proof for normally hyperbolic invariant manifolds},  J. Diff. Eq., 259(2015) 6215--6286.

\bibitem{Meln} M.J. Capi\'nski, P. Zgliczy\'nski, \emph{Beyond the Melnikov method: A computer assisted approach},
J. Diff. Eq., 262 (2017) 365--417.










\bibitem{FH} J-L. Figueras, A. Haro, \emph{Reliable computation of robust response tori on the verge of breakdown}. SIAM J. Appl. Dyn. Syst. 11 (2012), no. 2, 597--628.





\bibitem{LJR} J-P. Lessard, J.D. Mireles James, C. Reinhardt, \emph{Computer assisted proof of transverse saddle-to-saddle connecting orbits for first order vector fields}. J. Dynam. Differential Equations 26 (2014), no. 2, 267--313.




\bibitem{JM} J. D. Mireles James, K. Mischaikow, \emph{Rigorous a posteriori computation of (un)stable manifolds and connecting orbits for analytic maps. }SIAM J. Appl. Dyn. Syst. 12 (2013), no. 2, 957--1006.



%

\bibitem {HM} P. J. Holmes, J. E. Marsden. \textit{Melnikov's method and Arnold diffusion for perturbations of integrable Hamiltonian systems}. J. Math. Phys. 23 (1982), no. 4, 669--675.

\bibitem{Lerman} L. M. Lerman, Ia. L. Umanski\u{\i},
\emph{On the existence of separatrix loops in four-dimensional systems similar to the integrable Hamiltonian systems.}
J. Appl. Math. Mech. 47 (1983), no. 3, 335--340 (1984); translated from
Prikl. Mat. Mekh. 47 (1983), no. 3, 395--401

\bibitem{M} Melnikov V.K   {\em On the stability of the center for time periodic perturbations.} Trans. Moscow Math. Soc. 12 (1963), 1--57.

\bibitem{Mo} R.E. Moore, {\em Interval Analysis.} Prentice
Hall, Englewood Cliffs, N.J., 1966
%


\bibitem{N} A. Neumeier, {\em Interval methods for systems of
equations}, Cambridge University Press, 1990.


\bibitem{Rump} S. M. Rump, \emph{Verification methods: Rigorous results using floating-point arithmetic}, Acta Numer., 19
(2010), pp. 287--449.

\bibitem{Sch} J.T. Schwartz, {\em Nonlinear Functional Analysis}, 1969,
Gordon and Breach Science Publishers Inc., New York.

\bibitem{Tre94} D.V. Treschev,  \emph{Hyperbolic tori and asymptotic surfaces in Hamiltonian systems}, Russ. J. Math.
Phys. 2 (1), 93–110 (1994)

\bibitem{Wiggins} S. Wiggins, \emph{Global bifurcations and chaos.
Analytical methods.} Applied Mathematical Sciences, 73. Springer-Verlag, New York, 1988.

\bibitem{Wi} S. Wiggins. \newblock {\em Normally hyperbolic invariant manifolds in dynamical systems}, volume 105 of \emph{Applied Mathematical Sciences}. \newblock Springer-Verlag, New York, 1994.

\bibitem{WZ} D. Wilczak and P. Zgliczy\'nski, \emph{Cr-Lohner algorithm}, Schedae Informaticae, 20 (2011), pp. 9--46.

 \bibitem{Z} P. Zgliczy\'nski, \emph{C1 Lohner algorithm}. Found. Comput. Math. 2 (2002), no. 4, 429--465.

\end{thebibliography}
\end{document}